\documentclass{imanum}
\crefname{equation}{}{}
\usepackage{graphicx}
\usepackage{autonum}
\usepackage{mathtools}
\usepackage{mathrsfs}
\DeclareMathAlphabet{\mathcal}{OMS}{cmsy}{m}{n}
\usepackage{braket}

\renewcommand{\u}{\vec{u}}

\newcommand{\x}{\vec{x}}
\newcommand{\z}{\vec{z}}

\newcommand{\V}{\vec{V}}
\newcommand{\w}{\vec{w}}

\renewcommand{\v}{\vec{v}}
\newcommand{\g}{\vec{g}}

\newcommand{\f}{\vec{f}}

\newcommand{\N}{\mathbb{N}}

\newcommand{\X}{\vec{X}}
\newcommand{\Th}{\mathcal{T}_h}

\renewcommand{\P}{\mathcal{P}}

\newcommand{\Om}{\Omega}

\newcommand{\finaltime}{\mathscr{T}}

\DeclareMathOperator{\supp}{supp}
\DeclarePairedDelimiter{\abs}{\lvert}{\rvert}
\DeclarePairedDelimiter{\norm}{\lVert}{\rVert}

\DeclarePairedDelimiter{\twonorm}{\lVert}{\rVert_{L^2(\Omega)}}
\DeclarePairedDelimiter{\stwonorm}{\lVert}{\rVert^2_{L^2(\Omega)}}

\newtheorem{assumption}[theorem]{\sc Assumption}
\crefname{assumption}{Assumption}{Assumptions}
\crefname{proposition}{Proposition}{Propositions}

\jno{drnxxx}
\received{-}
\revised{-}

\begin{document}

\title{Fully discrete best approximation type estimates in $L^{\infty}(I;L^2(\Omega)^d)$ for finite element discretizations of the transient Stokes equations}
\shorttitle{$L^{\infty}(I;L^2(\Omega))$ Stokes best approximation estimates}

\author{%
    {\sc
	Niklas Behringer\thanks{Email: niklas.behringer@tum.de},
    Boris Vexler\thanks{Corresponding author. Email: vexler@ma.tum.de}}\\[2pt]
    Chair of Optimal Control, Center for Mathematical Sciences, \\ Technical University of Munich, 85748 Garching by Munich, Germany\\[6pt]
    {\sc and}\\[6pt]
    {\sc Dmitriy Leykekhman}\thanks{Email: dmitriy.leykekhman@uconn.edu}\\[2pt]
    Department of Mathematics, University of Connecticut, Storrs, CT 06269
}
\shortauthorlist{N. Behringer \emph{et al.}}

\maketitle

\begin{abstract}
    {
	In this article we obtain an optimal best approximation type result for fully discrete approximations of the transient Stokes problem. For the time discretization we use the discontinuous Galerkin method and for the spatial discretization we use standard finite elements for the Stokes problem satisfying the discrete inf-sup condition. The analysis uses  the technique of discrete maximal parabolic regularity. The results require only natural assumptions on the data and do not assume any additional smoothness of the solutions. 
    }
    {
	transient Stokes, discontinuous Galerkin method, finite elements, best approximation, pointwise error estimates, a priori estimates
    }
\end{abstract}

\section{Introduction}
In this paper we consider the following transient Stokes problem with no-slip boundary conditions, \sbox0{\cref{chap:IS:eq:istokes_1,chap:IS:eq:istokes_2,chap:IS:eq:istokes_3,chap:IS:eq:istokes_4}}
\begin{subequations}\label{eq:transient:Stokes}
    \begin{alignat}{3}
	\partial_t\u-\Delta \u+ \nabla p &= \f \quad &&\text{in }I\times\Omega, \label{chap:IS:eq:istokes_1}\\
	\nabla \cdot \u &= 0 \quad &&\text{in } I\times\Omega, \label{chap:IS:eq:istokes_2}\\
	\u &= \vec 0 \quad  &&\text{on } I\times\partial \Omega, \label{chap:IS:eq:istokes_3}\\
	\u(0) &= \u_0 &&\text{in } \Omega.\label{chap:IS:eq:istokes_4}
    \end{alignat}
\end{subequations}
Throughout this work, we assume that $\Omega\subset \mathbb{R}^d$, $d\in \{2,3\}$, is a bounded polygonal/polyhedral Lipschitz domain, $\finaltime>0$ and $I=(0,\finaltime]$.
We will require some (weak) assumptions on the data, which essentially allow for a weak formulation including both velocity and pressure and for $\u \in C(\bar I;L^2(\Omega)^d)$.
We consider fully discrete approximations of problem \cref{eq:transient:Stokes}, where we use compatible finite elements (i.e. satisfying a uniform inf-sup condition) for the space discretization and the discontinuous Galerkin method for the temporal discretization. 
Our goal is to obtain best approximation type results, that do not involve any additional regularity assumptions on the solution beyond the regularity which follows  directly from the assumed data above.
Such results are important in the analysis of PDE constrained optimal control problems that we have in mind. We refer, e.g., to \cite{2011Meidner}, where such estimates  are required for numerical analysis of an optimal control problem constrained by the heat equation with state constraints pointwise in time.

Our main result is of the following form
\begin{equation}\label{eq:intro:best_approx}
    \norm{\u-\u_{{\tau}h}}_{L^{\infty}(I;L^2(\Omega))} \leq C\ell_\tau  \Big(\norm{\u-\v_{{\tau}h}}_{L^{\infty}(I;L^2(\Omega))} +  \norm{\u - R_h^S(\u,p)}_{L^{\infty}(I;L^2(\Omega))} \Big), %
\end{equation}
where $\u_{{\tau}h}$ is the fully discrete finite element approximation of the velocity $\u$, $\v_{{\tau}h}$ is an arbitrary  function from the finite element approximation of the velocity spaces $X^w_{\tau}(\V_h)$, $R_h^S$ is the Ritz projection for the stationary Stokes problem and $\ell_\tau$ is a logarithmic term, explicitly given in the statements of the results, see \cref{chap:IS:corr_best_approx}.

The result \cref{eq:intro:best_approx} links the approximation error for the fully discrete transient Stokes problem to the best possible approximation of a continuous solution $\u$ in the discrete space $X^w_{\tau}(\V_h)$ as well as the approximation of the stationary Stokes problem in $\V_h$. Such results go in hand with only natural assumptions on the problem data and thus are desirable in applications. For this result we do not require additional regularity of the domain allowing, e.g., for reentrant corners and edges. Moreover, we do not require the mesh to be quasi-uniform nor shape regular. Therefore, the result is also true for graded and even anisotropic meshes (provided the discrete inf-sup condition holds uniformly on such meshes). The application of \cref{eq:intro:best_approx} in such cases would require corresponding results for the stationary Stokes problem to estimate $\u - R_h^S(\u,p)$, see \cref{remark:non_convex}.

Under the additional assumption of convexity of $\Omega$ and some approximation properties of the discrete spaces we prove error estimates of the form
\[
\norm{\u-\u_{{\tau}h}}_{L^{\infty}(I;L^2(\Omega))}\nonumber
\leq C\ell_\tau \left({\tau} + h^2\right) \left( \norm{\f}_{L^{\infty}(I;L^2(\Omega))} + \norm{\u_0}_{\V^2}\right),
\]
where $\V^2$ is an appropriate space introduced in the next section. This estimate seems to be optimal (probably up to logarithmic terms) with respect to both the assumed regularity of the data and the order of convergence.

In the case of the heat equation, a similar estimate with respect to $L^\infty(I;L^2(\Omega))$ is derived in \cite{2011Meidner} and for a non-autonomous parabolic problem in \cite[Theorem 4.5]{2018LeVe}. For corresponding estimates in the maximum norm in the case of the heat equation we refer to \cite{1995Eriksson, 2016Leykekhmana, 1980Schatza,2011Meidner}  and for the maximum norm of the gradient to \cite{2017Leykekhmana,2008Leykekhmana,1989Thomee}. Further results are also available in case of discretization only in space. For an overview and respective references we refer to \cite{2016Leykekhmana,2017Leykekhmana}.

We are not aware of any best approximation max-norm estimates in time \emph{and} space for the instationary Stokes problem \cref{eq:transient:Stokes} in the literature. 
A result for the fully discrete problem in form of $L^{\infty}(I;L^2(\Omega)^d)$  estimates based on discontinuous Galerkin methods is provided in \cite{2010Chrysafinos}, including an overview over related results for (semi-)discrete problems based on other discretization approaches.
Recently the numerical behavior of a stabilized discontinuous Galerkin scheme for the Stokes problem has been analyzed in \cite{2017Ahmed}.
Furthermore, there are results for the fully discrete Navier-Stokes problem under moderate regularity assumptions in \cite{1990Heywood}.
Here, we focus on an approach via a discontinuous Galerkin time stepping scheme similar to the approach in \cite{2010Chrysafinos,2017Leykekhman}.
However, all the results mentioned above differ from ours in an essential way. We give a more detailed comparison of our result and the existing results from the literature in \cref{sec:error_estimates}. 

Our main technical tools are continuous and discrete maximal parabolic regularity results. On the continuous level we use the estimate
\[
	\norm{\partial_t \u}_{L^s(I; L^2(\Omega))} + \norm{A \u}_{L^s(I;L^2(\Omega))} + \norm{p}_{L^s(I;L^2(\Omega))}  \leq \frac{C s^2}{s-1} \norm{\f}_{L^s(I; L^2(\Omega))}
\] 
for $\f \in L^s(I; L^2(\Omega)^d)$, $\u_0 = 0$, $1<s<\infty$ and $A$ being the Stokes operator \cref{eq:StokesOperator}, see \cref{proposition:max_reg} and \cref{theorem:weak_with_preasure} for the details and also for the formulation in the case $\u_0 \neq 0$. This estimate holds on a general Lipschitz domain $\Omega$. Assuming in addition the convexity of $\Omega$, we have
\[
	\norm{\partial_t \u}_{L^s(I; L^2(\Omega))} + \norm{\u}_{L^s(I;H^2(\Omega))} + \norm{\nabla p}_{L^s(I;L^2(\Omega))}  \leq \frac{C s^2}{s-1} \norm{\f}_{L^s(I; L^2(\Omega))},
\]
see \cref{remark:Omega_convex_1} and \cref{cor:Omega_convex_2}. On the discrete level, we provide the corresponding estimates which hold even in the limit cases $s=1$ and $s=\infty$ at the expense of an logarithmic term. In a way, we extend the discrete maximal parabolic regularity results from \cite{2017Leykekhman} to the Stokes problem. The resulting estimate is 
\[
	\norm{\partial_t \u_{{\tau}h}}_{L^s(I; L^2(\Omega))} + \norm{A_h \u_{{\tau}h}}_{L^s(I;L^2(\Omega))} \leq C \ln \frac{\finaltime}{{\tau}} \norm{\f}_{L^s(I; L^2(\Omega))},
\]
where $A_h$ is the discrete Stokes operator, see \cref{chap:IS:corollary:maximal_regularity_discrete} for details and the precise formulation. Under the convexity assumption for the domain $\Omega$, similar to the continuous case, we also obtain
\[
\norm{\Delta_h \u_{{\tau}h}}_{L^s(I;L^2(\Omega))} + \norm{\nabla p_{{\tau}h}}_{L^s(I;L^2(\Omega))} \le \ln \frac{\finaltime}{{\tau}} \norm{\f}_{L^s(I; L^2(\Omega))},
\]
where $\Delta_h$ is the discrete Laplace operator, see \cref{remark:delta_est} and \cref{DMPR:Pressure} for details.

In the next section we introduce a framework of function spaces for the treatment of the stationary and transient Stokes problem, the Stokes operator and the resolvent problem. Moreover, we discuss the weak formulation and the regularity issues for \cref{eq:transient:Stokes}.
In \cref{sec:discretization} we discuss the spatial discretization, introduce respective discrete spaces, operators and prove a discrete resolvent estimate. In \cref{chap:IS:section:discontG} we present
a full discretization of \cref{eq:transient:Stokes} and show discrete smoothing and discrete maximal regularity results for the velocity in \cref{sec:maximal_regularity} based on the operator calculus discussed for the heat equation in \cite{2017Leykekhman}. This allows us to prove best approximation results for the velocity in \cref{sec:bestapproximation}. In \cref{sec:error_estimates} we apply the best approximation type results to prove error estimates and compare these result to the existing results in the literature. Finally in \cref{sec:pressure} we explore an expansion of the discrete maximal parabolic estimates to the pressure.

\section{Results on the continuous level}

In this section, we introduce the function spaces we require for the analysis of \cref{eq:transient:Stokes} and state some of the main properties of these spaces. In the later sections we adopt a technique based on discrete maximal parabolic regularity  from \cite{2017Leykekhman}, where we used an operator calculus for $-\Delta$ and its finite element analog $-\Delta_h$. In order to modify the corresponding results, we will introduce the continuous and the discrete Stokes operator. Furthermore, we will require analysis for the resolvent of these operators.  In our presentation  we follow the notation and presentation of \cite[Section~1 and Section~2]{2008Guermond}.

\subsection{Function spaces and Stokes operator}
In the following,
we will use the usual notation to denote the Lebesgue spaces $L^p$ and Sobolev spaces $H^k$ and $W^{k,p}$. The space  $L^2_0(\Om)$ will denote a subspace of $L^2(\Om)$  with mean-zero functions.  The inner product on $L^2(\Omega)$ as well as on $L^2(\Omega)^d$ is denoted by $(\cdot,\cdot)$. To improve readability, we omit the superscript $d$ when having for example $L^2(\Omega)^d$ appear as subscript to norms.
We also introduce 
the following function spaces
\begin{equation}
    \mathcal{V} = \Set{\v \in C^{\infty}_0(\Omega)^d | \nabla \cdot \v =0},\quad
    \V^0 =  \overline{\mathcal{V}}^{L^2},\quad    \V^1 =  \overline{\mathcal{V}}^{H^{1}},%
\end{equation}
where the notation in the last line denotes the completion of the space $\mathcal{V}$ with respect to the $L^2(\Omega)^d$ and $H^1(\Omega)^d$ topology, respectively. Notice that functions in $\V^1$ %
have zero boundary conditions in the trace sense.
Alternatively we have  
$$
\V^1= \set{\v \in H_0^{1}(\Om)^d | \nabla \cdot \v = 0}
$$
by \cite[Theorem III.4.1]{2011Galdi}.

We define the vector-valued Laplace operator
$$
-\Delta \colon D(\Delta) \rightarrow L^2(\Omega)^d,
$$
where the domain $D(\Delta)$ is understood with respect  to $L^2(\Omega)^d$ and is given as
\[
    D(\Delta) = \set{\v \in H_0^{1}(\Om)^d | \Delta \v \in L^2(\Omega)}.
\]
If the domain $\Omega$ is  convex, then the standard $H^2(\Omega)$ regularity implies $ D(\Delta) = H^1_0(\Omega)^d \cap H^2(\Omega)^d$. 
In addition, we introduce the space $\V^2$ as
\[
    \V^2 = \V^1 \cap D(\Delta).
\]
We will also use the following Helmholtz decomposition (cf. \cite[Chapter I, Theorem 1.4]{1977Temam} and \cite[Theorem III.1.1]{2011Galdi})
\begin{equation}
    L^2(\Omega)^d = \V^0 \oplus \nabla \left(H^{1}(\Omega) \cap L^2_0(\Omega)\right). \label{chap:IS:eq:helmholtz_decomposition}
\end{equation}
As usual we define the Helmholtz projection $\mathbb{P} \colon L^2(\Omega)^d \rightarrow \V^{0}$ (often called the Leray  projection) as the $L^2$-projection from $L^2(\Om)^d$ onto $\V^{0}$. %
Using $\mathbb{P}$ and $-\Delta$, we define the Stokes operator $A\colon \V^2 \rightarrow \V^0$ as
\begin{equation}\label{eq:StokesOperator}
    A = - \mathbb{P} \Delta \vert_{\V^2}.
\end{equation}
The operator $A$ is a self adjoint, densely defined and positive definite operator on $\V^0$. We note that $D(A)=\V^2$. There holds
\[
    (A \v,\v) = \|\nabla \v\|^2_{L^2(\Om)} \ge \lambda_0 \|\v\|^2_{L^2(\Om)} \quad \v \in \V^2,
\]
where $\lambda_0>0$ is the smallest eigenvalue of the Laplace operator $-\Delta$ given by
\begin{equation}\label{eq:lambda_0}
    \lambda_0 = \inf_{v \in H^1_0(\Omega)} \frac{\|\nabla v\|^2_{L^2(\Omega)}}{\|v\|^2_{L^2(\Omega)}}.
\end{equation}
Similar to the Laplace operator, for convex polyhedral domains $\Om$ we have the following $H^2$ regularity bound due to
\cite{1989Dauge,1976Kellogg}
\begin{equation}\label{eq: H2 regularity A}
    \|\vec{v}\|_{H^2(\Om)}\le C\|A\vec{v}\|_{L^2(\Om)},\quad \forall \vec{v}\in \V^2.
\end{equation}

\subsection{Stokes resolvent problem}

The key to our analysis is the spectral representation of the semigroup generated by $A$. For that we consider the Stokes resolvent problem for $\f \in L^2(\Omega)^d$\sbox0{\cref{cha:eq:stokes_resolvent_1,cha:eq:stokes_resolvent_2,cha:eq:stokes_resolvent_3}}
\begin{subequations}\label{eq:stokes_resolvent}
    \begin{alignat}{3}
	z\u -\Delta \u + \nabla p &= \f \quad &&\text{in }\Omega, \label{cha:eq:stokes_resolvent_1}\\
	\nabla \cdot \u &= 0 \quad &&\text{in } \Omega, \label{cha:eq:stokes_resolvent_2}\\
	\u &= \vec 0 \quad  &&\text{on } \partial \Omega. \label{cha:eq:stokes_resolvent_3}
    \end{alignat}
\end{subequations}
Here $z \in \Sigma_{\theta,\bar\omega}$, which is defined as
\begin{equation}
    \Sigma_{\theta,\bar\omega}=\set{ c \in \mathbb{C} | c \neq \bar\omega \text{ and } \abs{\arg(c - \bar\omega)} < \theta }.
\end{equation}
The solution $(\u,p)$ to \cref{eq:stokes_resolvent} is a complex valued function in $H_0^1(\Omega)^d \times L^2_0(\Omega)$ as complex valued function spaces with a hermitian inner product.
In our situation we are interested in the case of 
$\theta \in (\pi/2,\pi)$ and $\bar\omega \in [-\lambda_0,0]$ with $\lambda_0>0$ from \cref{eq:lambda_0}. 
\begin{proposition}\label{prop:A:sectorial}
    The operator $A$ is sectorial. In particular, for every $\theta \in (\pi/2,\pi)$ there exists a constant $C = C_\theta$ such that for all $z \in \Sigma_{\theta,\bar\omega}$ with $\bar\omega \in [-\lambda_0,0]$ and $(\u,p)$ being the solution of \cref{eq:stokes_resolvent} with $\f\in L^2(\Omega)^d$ there holds the following resolvent estimate
    \begin{equation}\label{eq: cont. resolvent in L2}
	\norm{\u}_{L^2(\Omega)} \leq \frac{C }{\abs{z- \bar\omega}} \norm{\mathbb{P}\f}_{L^2(\Omega)}.
    \end{equation}
\end{proposition}
\begin{proof}
    It is straightforward to check, that every self-adjoint positive operator, which is densely defined on a Hilbert space is sectorial. This applies to the operator $A - \bar \omega \operatorname{Id}$ for all $\bar\omega \in [-\lambda_0,0]$, which results in the resolvent estimate \cref{eq: cont. resolvent in L2}. Below we provide a direct proof for the discrete version of the operator $A$, see \cref{lemma:l2_resolvent_discrete}, which is also applicable here. 
\end{proof}

Using the Stokes operator $A$ \cref{eq:StokesOperator} one can rewrite the resolvent estimate \cref{eq: cont. resolvent in L2} as 
\begin{equation}\label{eq: cont. resolvent in L2 using A}
    \norm{(z+A)^{-1}\mathbb{P}\f}_{L^2(\Omega)} \leq \frac{C }{\abs{z- \bar\omega}} \norm{\mathbb{P}\f}_{L^2(\Omega)}.
\end{equation}
\begin{remark}
    The resolvent estimates in $L^p$ norms for $p\neq 2$ are also known. For example, for $d=3$ on Lipschitz domains, \cite{2012Shen} has shown a resolvent estimate for some  interval of $p$ satisfying  $\abs*{ {1}/{p} - {1}/{2} } < {1}/{6}  + \varepsilon$,
    for $\varepsilon>0$.
    On smooth $C^3$ domains it is known to hold even for $p=\infty$ (cf. \cite{2015Abe}). However, the extension to non-smooth convex domains is still an open problem and it even appears in a collection of open problems (cf.~\cite[Problem 66]{2018Mazya}).
\end{remark}

\subsection{Weak formulation and regularity}
In this section we discuss the weak formulation and the regularity  of the transient Stokes problem \eqref{eq:transient:Stokes}. We will use the notation $L^s(I;X)$ for the corresponding Bochner space with a Banach space $X$. Moreover, we will use also the standard notation $H^1(I;X)$. The inner product in $L^2(I;L^2(\Omega)^d)$ and in $L^2(I;L^2(\Omega)^d)$ is denoted by $(\cdot,\cdot)_{I \times \Omega}$. We will also use the notation $(f,g)_{I \times \Omega}$ for the corresponding integral for $f \in L^s(I;L^2(\Omega)^d)$ and $g \in L^{s'}(I;L^2(\Omega)^d)$ with $1\le s \le \infty$ and the dual exponent $s'$.

For the application of Galerkin finite element methods in space and time we will require a space-time weak formulation of the transient Stokes equations with respect to both variables, velocity $\u$ and pressure $p$. In a standard variational setting, e.g., with $f \in L^2(I;(\V^1)')$ or $f \in L^1(I;L^2(\Omega)^d)$, this is not possible, since only distributional pressure can be expected in general, see \cref{remark:distributional_pressure} below. Therefore, we will first introduce the (standard) weak formulation on the divergence free space, then we discuss regularity issues and introduce a velocity-pressure weak formulation based on a slightly stronger assumption on the data.
\begin{proposition}\label{prop:weak_solution}
    Let $\f \in L^1(I;L^2(\Omega)^d)$ and $\u_0 \in \V^0$. Then there exists a unique solution $\u \in L^2(I;\V^1) \cap C(\bar I,\V^0)$ with $\partial_t \u \in L^1(I;\V^0) + L^2(I;(\V^1)')$ fulfilling
    $\u(0) = \u_0$ and
    \begin{equation}\label{eq: weak stokes}
	\langle \partial_t \u,\v\rangle + (\nabla \u,\nabla \v)_{I \times \Omega} = ( \f,\v)_{I \times \Omega} \quad \text{for all } \v \in L^2(I;\V^1)\cap L^\infty(I;\V^0).
    \end{equation}
    Here, $\langle \partial_t \u,\v \rangle$ for $\partial_t \u = \w_1 + \w_2 \in L^1(I;\V^0) + L^2(I;(\V^1)')$ and $\v \in  L^2(I;\V^1)\cap L^\infty(I;\V^0)$ is understood as
    \[
	\langle \partial_t \u,\v\rangle = (\w_1,\v)_{I \times \Omega} + \langle \w_2 ,\v\rangle_{L^2(I;(\V^1)') \times L^2(I;\V^1)}.
    \]
    There holds the estimate
    \[
	\norm{\nabla \u}_{L^2(I;L^2(\Omega))} + \norm{\u}_{C(\bar I;L^2(\Omega))} \le C \left(\norm{\f}_{L^1(I;L^2(\Omega))} + \norm{\u_0}_{\V^0}\right).
    \]
\end{proposition}
\begin{proof}
    For the existence of the solution under the stated assumptions and with the corresponding regularity we refer to \cite[Chapter III, Theorem 1.1]{1977Temam} and its extension to the case $\f \in L^1(I;L^2(\Omega)^d)$ on page 264. The notion of the solution in \cite{1977Temam} is formulated in the almost everywhere sense on $I$, from which the formulation \cref{eq: weak stokes} follows directly by integration in time. The uniqueness of $\u$ solving \cref{eq: weak stokes} is also obtained in the standard manner choosing $\v = \u$ for $\f=0$ and $\u_0 = 0$.
\end{proof}

\begin{remark}\label{remark:fl2v*}
    Another possibility to formulate the notion of the weak solution is to assume $\f \in L^2(I; (\V^1)')$. Since we require in the sequel the additional assumption $\f \in L^s(I;L^2(\Omega)^d)$ for some $s>1$ we prefer to use the formulation from \cref{prop:weak_solution}.
\end{remark}
\begin{remark}\label{remark:distributional_pressure}
    Under the assumptions from \cref{prop:weak_solution} the existence of the corresponding pressure $p$ can be shown only in the following distributional sense. There exists $P \in C(\bar I,L^2(\Omega))$ such that the Stokes system holds in the distributional sense for $\u$ solving \cref{eq: weak stokes} and $p = \partial_t P$. Especially one can not expect in general $p \in L^1(I \times \Omega)$, cf. \cite[Chapter III, p. 267]{1977Temam}.
\end{remark}

In the following, we will discuss some additional regularity for the solution. On the one hand we need a slightly more regularity in order to be able to introduce the pressure $p$ as a function, cf. \cref{remark:distributional_pressure}. Moreover, additional regularity beyond $\u \in C(\bar I;L^2(\Omega)^d)$ is required if we use the best approximation result \cref{eq:intro:best_approx} for providing 
(optimal) error estimates, see \cref{sec:error_estimates}. 
It is well known, cf. again \cite{1977Temam}, that in the sense of \cref{prop:weak_solution} the equation \cref{eq: weak stokes} can be understood as an abstract parabolic problem
\begin{equation}\label{eq:stokes_abstract_equation}
    \begin{aligned}
	\partial_t  \u + A \u &= \mathbb{P}\f\qquad \text{for a.a. } t \in I,\\
	\u(0) &= \u_0,
    \end{aligned}	
\end{equation}
with the Stokes operator $A$ defined in \cref{eq:StokesOperator}.

Furthermore, note that by \cref{prop:A:sectorial}, the operator $A$ is sectorial and thus a generator of an analytic semigroup \cite[Definition 2.0.1, 2.0.2]{1995Lunardi}. 
For the Hilbert space setting it has then been shown in \cite{1964Simon} that this is equivalent to having a maximal regularity estimate of the form
\begin{equation}\label{eq:max_par_reg_1}
    \norm{\partial_t \u}_{L^s(I; L^2(\Omega))} + \norm{A \u}_{L^s(I;L^2(\Omega))}  \leq C_s \norm{\f}_{L^s(I; L^2(\Omega))}
\end{equation}
for problem \cref{eq:transient:Stokes} with $\u_0=0$, $1< s < \infty$ and $\f \in L^s(I;L^2(\Omega)^d)$.
For more details, we refer to \cite[Chapter IV, Theorem 1.6.3]{2014Sohr}.
In \cref{sec:maximal_regularity} we derive a respective estimate for a fully discrete version of \cref{eq:max_par_reg_1}, a so called discrete maximal parabolic regularity result based on ideas from \cite{2017Leykekhman}. 

For proving (optimal) error estimates in \cref{sec:error_estimates} we will use the maximal parabolic estimate \cref{eq:max_par_reg_1} for $s \to \infty$. To this end we will need precise dependence of  the constant $C_s$ on $s$ from \cite[Chapter 1, eq. (3.9), Theorem 3.2]{1994Ashyralyev}. Moreover, we require this regularity result also for the case of non-homogeneous initial conditions. %

To state this result, we consider the space of initial data $\V^0_{1-\frac{1}{s}}$ for $1 < s < \infty$ as in \cite[Chapter 1, Section 3.3]{1994Ashyralyev}. The Banach space $\V^0_{1-\frac{1}{s}}$ with the norm 
\begin{equation}\label{def:interpolation_space}
\norm{\v_0}_{\V^0_{1-\frac{1}{s}}} = \left( \int_0^1 \norm{A\exp(-t A) \v_0}_{\V^0}^s dt \right)^{1/s} + \norm{\v_0}_{\V^0} 
\end{equation}
contains all functions $\v_0\in \V^0$ such that for a solution $\u$ to the transient Stokes problem with right-hand side $\f=0$ and initial data $\u_0=\v_0$ it holds $A\u\in L^s(I;\V^0)$.

\begin{proposition}\label{proposition:max_reg}
    Let $1 < s < \infty$, $\f \in L^s(I;L^2(\Omega)^d)$ and $\u_0 \in \V^0_{1-\frac{1}{s}}$.%
    Then the solution $\u$ to the problem
    \begin{align}
	\partial_t  \u + A \u &= \mathbb{P}\f\qquad \text{for a.a. } t \in I,\\
	\u(\x,0) &= \u_0,
    \end{align}
    fulfills $\partial_t \u, A \u \in L^s(I; L^2(\Omega)^d)$. Moreover, there is a constant $C$ independent on $s$, $f$ and $u_0$ such that
    \begin{equation}
	\norm{\partial_t \u}_{L^s(I; L^2(\Omega))} + \norm{A \u}_{L^s(I;L^2(\Omega))}  \leq \frac{C s^2}{s-1} \Big(\norm{\f}_{L^s(I; L^2(\Omega))} + \norm{\u_0}_{\V^0_{1-\frac{1}{s}}}\Big).
    \end{equation}
\end{proposition}
\begin{proof}
    This follows from \cite[Chapter 1, Theorems 3.2, 3.7]{1994Ashyralyev} since $A$ is the generator of an analytic semigroup.
\end{proof}

\begin{remark}\label{remark:Omega_convex_1}
    If the domain is polyhedral/polygonal and convex then \cref{proposition:max_reg} provides $\u \in L^s(I;H^2(\Omega)^d)$ and the estimate
    \[
	\norm{\partial_t \u}_{L^s(I; L^2(\Omega))} + \norm{\u}_{L^s(I;H^2(\Omega))}  \leq \frac{C s^2}{s-1} \Big(\norm{\f}_{L^s(I; L^2(\Omega))} + \norm{ \u_0}_{\V^0_{1-\frac{1}{s}}}\Big).
    \]
\end{remark}

\begin{remark}\label{remark:v1}
	There holds $\V^1 \hookrightarrow \V^0_{\frac{1}{2}}$. This follows from the fact that for the homogeneous problem ($\f=0$) with $\u_0 \in \V^1$ the following estimate holds
	\[
	 \norm{A \u}_{L^2(I;L^2(\Omega))} \le 2 \norm{A^{\frac{1}{2}}\u_0}_{L^2(\Omega)} = 2 \norm{\u_0}_{\V^1}.
	\]
	 The above inequality is stated, e.g., in \cite[Chapter IV, Theorem 1.5.2]{2014Sohr}. For the representation of the norm on $\V^1$ by $A{^\frac{1}{2}}$ see, e.g.,  \cite[Chapter III, Lemma 2.2.1]{2014Sohr}. Therefore, for the range $1 < s \le 2$ it is sufficient to assume $\u_0 \in \V^1$ for the estimate in \cref{proposition:max_reg}.
\end{remark}

\begin{remark}\label{remark:s_to_infty}
    If $u_0 \in \V^2$ there holds for every $1 < s < \infty$
    \[
	\norm{\u_0}_{\V^0_{1-\frac{1}{s}}} \le \norm{A\u_0}_{L^2(\Omega)}.
    \]
    We can argue as follows. Since $\u_0 \in \V^2$, we have that $A$ commutes with $\exp(-t A)$ (cf. \cite[Chapter II, eq. (3.2.19)]{2014Sohr}) and due to the boundedness of $\exp(-t A)$ in the operator norm (cf. \cite[Chapter IV, eq. (1.5.8)]{2014Sohr}) we can conclude using definition \cref{def:interpolation_space}:
    \begin{align}
	\norm{ \u_0}_{\V^0_{1-\frac{1}{s}}} 
	&= \left( \int_0^1 \norm{A\exp(-t A) \u_0}_{\V^0}^s dt \right)^{1/s} + \norm{\u_0}_{\V^0}\\
	&\leq \left( \int_0^1 \norm{\exp(-t A)}^s_{\V^0 \rightarrow \V^0} \norm{A\u_0}_{\V^0}^s dt \right)^{1/s} + \norm{\u_0}_{\V^0} \leq C \norm{A\u_0}_{L^2(\Omega)}.
    \end{align}
    Therefore we have the following version of the maximal parabolic regularity estimate
    \[
	\norm{\partial_t \u}_{L^s(I; L^2(\Omega))} + \norm{A \u}_{L^s(I;L^2(\Omega))}  \leq \frac{C s^2}{s-1} \Big(\norm{\f}_{L^s(I; L^2(\Omega))} + \norm{A\u_0}_{L^2(\Omega)}\Big).
    \]
    which we will use in particular for $s \to \infty$.
\end{remark}

The next theorem provides the space-time weak formulation in both variables, velocity and pressure. Please note, that no additional regularity of the domain is required and the assumption $f \in L^1(I;L^2(\Omega)^d)$ from \cref{prop:weak_solution} is only slightly strengthened to $\f \in L^s(I;L^2(\Omega)^d)$ for some $s>1$.  

\begin{theorem}\label{theorem:weak_with_preasure}
    Let $\f \in L^s(I;L^2(\Omega)^d)$ for some $1<s<\infty$ and $\u_0 \in \V^0_{1-\frac{1}{s}}$.%
    Then there exists a unique solution $(\u,p)$ with
    \[
	\u \in L^2(I;\V^1)\cap C(\bar I,\V^0), \; \partial_t \u, A \u \in L^s(I; L^2(\Omega)^d) \quad \text{and} \quad p \in L^s(I;L^2_0(\Omega))
    \]
    fulfilling $\u(0)=\u_0$ and
    \begin{equation}\label{eq: weak stokes with pressure}
	(\partial_t \u,\v)_{I \times \Omega} + (\nabla \u,\nabla \v)_{I \times \Omega} - (p,\nabla \cdot \v)_{I \times \Omega} + (\nabla \cdot \u,\xi)_{I \times \Omega}= (\f,\v)_{I \times \Omega}
    \end{equation}
    for all
    \[
	\v \in L^2(I;H^1_0(\Omega)^d) \cap L^\infty(I;L^2(\Omega)^d)\quad \text{and}\quad \xi \in L^2(I;L^2_0(\Omega)).
    \]
    There holds the estimate
    \[
    \norm{p}_{L^s(I;L^2(\Omega))} \le \frac{C s^2}{s-1} \left( \norm{\vec{f}}_{L^s(I;L^2(\Omega))} + \norm{\u_0}_{\V^0_{1-\frac{1}{s}}}\right).
    \]
\end{theorem}
\begin{proof}
    We take the unique solution $\u$  with $\partial_t \u, A \u \in L^s(I; L^2(\Omega)^d)$ from \cref{proposition:max_reg} and $\u \in L^2(I;\V^1) \cap C(\bar I,\V_0)$ from \cref{prop:weak_solution}. To prove the existence of the corresponding pressure $p$ we consider for almost every $t \in I$ the element $\vec{g}(t) \in H^{-1}(\Omega)^d$,
    \[
	\langle \vec{g}(t),w\rangle = (\vec{f}(t),w) -(\partial_t \u(t),w) - (\nabla \u(t),\nabla w), \quad w \in H^1_0(\Omega)^d,
    \]
    i.e.,
    \[
	\vec{g}(t) = \vec{f}(t)-\partial_t \u(t) + \Delta \u(t) \in H^{-1}(\Omega)^d.
    \]
    This element is well defined due to $\vec{f}(t),\partial_t \u(t) \in L^2(\Omega)^d$ and $\u(t) \in \V^1 \hookrightarrow H^1_0(\Omega)$ for almost every $t \in I$. Moreover, there holds by \cref{eq: weak stokes}
    \[
	\langle \vec{g}(t),w\rangle = 0 \quad \forall w \in \V^1
    \]
    for almost all $t \in I$, since \cref{eq: weak stokes} holds also pointwise almost everywhere. Therefore, we can apply \cite[Chapter I, Proposition 1.1]{1977Temam}, which ensures the existence of a distribution $p(t)$ with
    \begin{equation}\label{eq:nabla_p}
	\nabla p(t) = \vec{g}(t)
    \end{equation}
    in the distributional sense for almost every $t \in I$. By \cite[Chapter I, Proposition 1.2]{1977Temam} we have $p(t) \in L^2_0(\Omega)$ and
    \[
	\norm{p(t)}_{L^2(\Omega)} \le C \norm{\vec{g}(t)}_{H^{-1}(\Omega)}.
    \]
    Using the definition of $\vec{g}$ we obtain $p \in L^s(I;L^2_0(\Omega))$ and
    \[
	\begin{aligned}
	    \norm{p}_{L^s(I;L^2(\Omega))} 
		\le C \norm{\vec{g}}_{L^s(I;H^{-1}(\Omega))}
		&\le C \left( \norm{\vec{f}}_{L^s(I;L^2(\Omega))} +\norm{\partial_t \u}_{L^s(I;L^2(\Omega))} + \norm{\nabla \u}_{L^s(I;L^2(\Omega))}\right)\\
		& \le \frac{C s^2}{s-1} \left( \norm{\vec{f}}_{L^s(I;L^2(\Omega))} + \norm{\u_0}_{\V^0_{1-\frac{1}{s}}}\right),
	\end{aligned}
    \]
    where we have used $\norm{\nabla \v}_{L^2(\Omega)} \le C \norm{A \v}_{L^2(\Omega)}$ for every $\v \in \V^1$ and \cref{proposition:max_reg}.
    With this regularity we obtain from \cref{eq:nabla_p} and the definition of $\vec{g}$
    \[
	(-p,\nabla \cdot \v)_{I \times \Omega} = (\vec{f},\v)_{I \times \Omega} -(\partial_t \u,\v)_{I \times \Omega} - (\nabla \u,\nabla \v)_{I \times \Omega}
    \]
    for all $\v \in L^2(I;H^1_0(\Omega)^d) \cap L^\infty(I;L^2(\Omega)^d)$. Furthermore, it holds
    \[
	(\nabla \cdot \u,\xi)_{I \times \Omega} = 0 \quad \forall \xi \in L^2(I;L^2_0(\Omega))
    \]
    by $\u \in L^2(I;\V^1)$. This results in the stated weak formulation.
\end{proof}
\begin{corollary}\label{cor:Omega_convex_2}
	Let the assumptions of \cref{theorem:weak_with_preasure} be fulfilled. Let in addition the domain $\Omega$ be convex. Then we have $p \in L^s(I, H^1(\Omega))$ and the corresponding estimate holds
    \[
    \norm{p}_{L^s(I;H^1(\Omega))} \le \frac{C s^2}{s-1} \left(\norm{\vec{f}}_{L^s(I;L^2(\Omega))} + \norm{\u_0}_{\V^0_{1-\frac{1}{s}}}\right).
    \]
\end{corollary}
\begin{proof}
By convexity of $\Omega$ we obtain $\u \in L^s(I;H^2(\Omega)^d)$ and the corresponding estimate, see \cref{remark:Omega_convex_1}. Then, we have
\[
\vec{g}(t) = \vec{f}(t)-\partial_t \u(t) + \Delta \u(t) \in L^2(\Omega)^d.
\]
for almost all $t \in I$ in the notation of the proof of \cref{theorem:weak_with_preasure}. This leads to the desired regularity and to the estimate.
\end{proof}

\begin{remark}
    For regularity beyond these estimates, we want to highlight \cite[Corollary 2.1]{1982Heywood}, where the authors show that bounds for, e.g., $\nabla^3 \u$, $\partial_{tt}\u$, go hand in hand with the need of the data $\u_0$, $\f$ and initial pressure $p_0$ (defined as $\lim_{t\to 0} p(t)$) satisfying a nonlocal compatibility condition for $t \rightarrow 0$ at the boundary, which is potentially difficult to verify.
\end{remark}

\section{Spatial discretization and discrete resolvent estimates}\label{sec:discretization}

In this section we consider the discrete version of the operators presented in the previous section.  

\subsection{Spatial discretization}
Let $\{\Th\}$\label{glos:triangulation} be a 
family of triangulations of $\bar \Omega$, consisting of closed simplices, where we denote by $h$\label{glos:h} the maximum mesh-size. Let $\X_h \subset H^1_0(\Omega)^d$\label{glos:fe_space_velocity} and $M_h \subset L^2_0(\Omega)$\label{glos:fe_space_pressure} be a pair of compatible finite element spaces, i.e., them satisfying a uniform discrete inf-sup condition,
\begin{equation}\label{chap02:eq:discrete_infsup}
    \sup_{\v_h \in \X_h}\frac{(q_h,\nabla \cdot \v_h) }{\twonorm{\nabla \v_h}} \geq \beta \twonorm{q_h} \quad \forall q_h \in M_h, 
\end{equation}%
with a constant $\beta >0$ independent of $h$. 
We introduce the usual discrete Laplace operator $-\Delta_h \colon \X_h \rightarrow \X_h$ by
\begin{equation}
    (-\Delta_h \z_h, \v_h) = (\nabla \z_h, \nabla \v_h), \qquad \forall \z_h,\v_h \in \X_h.
\end{equation}
To define a discrete version of the Stokes operator $A$, we first define the space of discretely divergence-free vectors $\V_h$ as
\begin{equation}\label{eq:discretely_divergence_Xh}
    \V_h = \Set{ \v_h \in \X_h | (\nabla \cdot \v_h, q_h) =0 \quad \forall q_h \in M_h}.
\end{equation}
Using this space we can define the discrete Leray projection $\mathbb{P}_h\colon L^1(\Om)^d\to \V_h$ to be the $L^2$-projection onto $\V_h$, i.e., 
\begin{equation}\label{eq:discrete_Leray_projection}
(\mathbb{P}_h \u, \v_h)=(\u,\v_h) \quad\forall \v_h \in \V_h.
\end{equation}
Using $\mathbb{P}_h$, we define the discrete Stokes operator $A_h\colon \V_h \rightarrow \V_h$ as $A_h = - \mathbb{P}_h \Delta_h\vert_{\V_h}$. By this definition we have that for $\u_h \in \V_h$, $A_h\u_h \in \V_h$ fulfills
\begin{equation}
    (A_h \u_h, \v_h) = (\nabla \u_h, \nabla \v_h), \qquad \forall \v_h \in \V_h.
\end{equation}
Notice, since $\V_h \subset \X_h$,   for $\v_h \in \V_h$ we obtain
\begin{equation}\label{chap:IS:eigenvalue}
    (A_h \v_h, \v_h) = (\nabla \v_h, \nabla \v_h)  \geq \lambda_0 \stwonorm{\v_h}
\end{equation}
where $\lambda_0$ is the smallest eigenvalue of $-\Delta$, see \cref{eq:lambda_0}.
This implies that the eigenvalues of $A_h$ are also positive and bounded from below by $\lambda_0$.

Moreover we define the orthogonal space $\V_h^\perp \subset \X_h$ as
\[
\V_h^\perp = \Set{\w_h \in X_h | (\w_h,\v_h) = 0 \quad \forall \v_h \in \V_h}.
\]
The following classical result, cf., e.g., \cite[Chapter II, Theorem 1.1]{1986Girault} will be used to provide existence and uniqueness of the fully discrete pressure in the sequel.
\begin{lemma}\label{lemma:stationary_existence_pressure}
For every $\w_h \in \V_h^\perp$ there exists a unique $p_h \in M_h$ such that
\[
(\w_h,\v_h) = (p_h, \nabla \cdot \v_h) \quad \forall \v_h \in \X_h.
\]
There holds
\[
\norm{p_h}_{L^2(\Omega)} \le \frac{1}{\beta} \norm{\nabla (-\Delta_h)^{-1} \w_h}_{L^2(\Omega)}.
\]
\end{lemma}
\begin{proof}
Note that we can decompose $\X_h = \V_h \oplus \V_h^\perp$ and there holds $\dim \V_h^\perp = \dim M_h$. The equation for $p_h$ can be then equivalently rewritten as
\[
p_h \in M_h \quad : \quad (p_h, \nabla \cdot \v_h) = (\w_h,\v_h) \quad \forall \v_h \in \V_h^\perp.
\]
The uniqueness of $p_h$ (as well as the estimate) follows then directly from the inf-sup condition \cref{chap02:eq:discrete_infsup}. The existence follows from uniqueness due to $\dim \V_h^\perp = \dim M_h$.
\end{proof} 

\subsection{Discrete resolvent estimate}
For a given $\f \in L^2(\Omega)^d$ the discrete version of the resolvent problem \cref{eq:stokes_resolvent} takes the form
\begin{equation}\label{eq: fe resolvent stokes}
    \u_h \in \V_h \quad:\quad z(\u_h,\v_h)+ (\nabla \u_h,\nabla \v_h) = (\f,\v_h) \quad \forall \v_h \in \V_h,
\end{equation}
which we can also write compactly using the discrete operator as
\begin{equation}\label{eq: fe resolvent stokes operator}
    (z+A_h)\u_h = \mathbb{P}_h\f.
\end{equation}
Next we establish the discrete resolvent estimate in the $L^2(\Om)^d$ norm, which is the discrete version of \cref{prop:A:sectorial}.
\begin{lemma}\label{lemma:l2_resolvent_discrete}
    For any $\theta \in (\pi/2,\pi)$ there exists a constant $C= C_{\theta}$ such that for any $\nu \in [0,\lambda_0]$ with $\lambda_0>0$ being the smallest eigenvalue of $-\Delta$, see \cref{eq:lambda_0}, it holds
    \begin{equation}
	\twonorm{\u_h}=\twonorm{(z+A_h)^{-1}\mathbb{P}_h\vec{f}} \leq \frac{C_{\theta}}{\abs{z+\nu}} \twonorm{\f} \quad \forall z \in \Sigma_{\theta,-\nu},
    \end{equation}
    where $\u_h \in \V_h$ is the solution to \cref{eq: fe resolvent stokes operator} with right-hand side $\f\in L^2(\Omega)^d$.
\end{lemma}
\begin{proof}
    Testing \cref{eq: fe resolvent stokes} with $\u_h$ we have
    \begin{equation}
	(z +\nu) \stwonorm{\u_h} + ((-\Delta_h - \nu)\u_h, \u_h)  = (\f,\u_h) , \label{chap:IS:aux_1}
    \end{equation}
    for any $\nu>0$.
    Since $-\Delta_h$ is positive definite with \cref{chap:IS:eigenvalue},  we have that $-\Delta_h-\nu$ is still a non-negative operator for $\nu \in [0,\lambda_0]$ and thus $((-\Delta_h-\nu)\u_h,\u_h)\geq 0$. Since $z$ is restricted to the sector $\Sigma_{\theta,-\nu}$ we can rewrite \cref{chap:IS:aux_1} as
    \begin{equation}
	\abs{z + \nu}e^{i\phi}\stwonorm{\u_h} + \delta = (\f,\u_h), \label{eq:phi_eq}
    \end{equation}
    where $\delta \geq 0$ and $\abs{\phi} < \theta$. If we multiply \cref{eq:phi_eq} with $e^{-i\phi/2}$, take the real part and use that $\cos(\theta/2)>0$, this results in
    \begin{equation}
	\abs{z + \nu}\stwonorm{\u_h} \leq \cos\left( \theta/2 \right)^{-1} \abs{(\f,\u_h)} = C_{\theta}\abs{(\f, \u_h)},
    \end{equation}
    which after an application of the Cauchy-Schwarz inequality completes the proof.
\end{proof}

\section{Temporal discretization: the discontinuous Galerkin method}
\label{chap:IS:section:discontG}
\leavevmode \\
In this section we introduce  the discontinuous Galerkin method for the time discretization of the transient Stokes equations, a similar method was considered, e.g., in \cite{2010Chrysafinos}.
For that, we partition $I = (0,\finaltime]$ into subintervals $I_m = (t_{m-1},t_m]$ of length ${\tau}_m = t_m - t_{m-1}$, where $0= t_0<t_1<\dots <t_{M-1}<t_M = \finaltime$. The maximal and minimal time steps are denoted by ${\tau}=\max_m {\tau}_m$ and ${\tau}_{\min}=\min_m{\tau}_m$, respectively.  The time partition fulfills the following assumptions:
\begin{enumerate}
    \item There are constants $C,\beta >0$ independent of ${\tau}$ such that
	\begin{equation}
	    {\tau}_{\min} \geq C {\tau}^{\beta}. 
	\end{equation}
    \item There is a constant $\kappa >0$ independent of ${\tau}$ such that for all\\$m=1, 2, \dots, M-1$
	\begin{equation}
	    \kappa^{-1} \leq \frac{{\tau}_m}{{\tau}_{m+1}} \leq \kappa.
	\end{equation}
    \item It holds ${\tau} \leq \frac{\finaltime}{4}$.
\end{enumerate}
For a given Banach space $\mathcal{B}$ and the order $w \in \N$ we define the semi-discrete space $X_{\tau}^w(\mathcal{B})$ of piecewise polynomial functions in time as
\begin{equation}
    X_{\tau}^w(\mathcal{B}) = \Set{ \v_{\tau} \in L^2(I; \mathcal{B}) | \v_{\tau}\vert_{I_m} \in \mathcal{P}_{w,I_m}(\mathcal{B}), m = 1,2, \dots, M},
\end{equation}
where $\mathcal{P}_{w,I_m}(\mathcal{B})$ is the space of polynomial functions of degree less or equal $w$ in time with values in $\mathcal{B}$, i.e.,
\begin{equation}\label{chap:IS:eq:polynomial}
    \mathcal{P}_{w,I_m}(\mathcal{B}) = \Set{ \v_{\tau} \colon I_m \rightarrow \mathcal{B} | \v_{\tau}(t) = \sum_{j=0}^w \v^j\phi_j(t) ,\;\v^j \in \mathcal{B}, j =0, \dots, w}. 
\end{equation}
Here, $\{\phi_j(t)\}$ is a polynomial basis in $t$ of the space $\mathcal{P}_w(I_m)$ of polynomials with degree less or equal $w$ over the interval $I_m$.
We use the following standard notation for a function $\u \in X_{\tau}^w(L^2(\Omega)^d)$%
\begin{equation}
\u_m^+ = \lim_{\varepsilon \rightarrow 0^+}\u(t_m+\varepsilon), \quad \u_m^- = \lim_{\varepsilon \rightarrow 0^+}\u(t_m-\varepsilon),\quad [\u]_m = \u_m^+-\u_m^-.
\end{equation}
For later use we introduce $P_\tau \colon L^2(I;L^2(\Omega)^d) \to X^w_{\tau}(L^2(\Omega^d))$ as the $L^2$ projection in time by
\begin{equation}\label{eq:P_tau}
(\v-P_\tau \v,\w_\tau)_{I \times \Omega} = 0 \quad \forall \w_\tau \in X^w_{\tau}(L^2(\Omega^d)).
\end{equation}
We  will use the following two standard properties
\begin{equation}\label{eq:P_tau_stability}
\norm{P_\tau \v}_{L^s(I,L^2(\Omega))} \le C \norm{\v}_{L^s(I,L^2(\Omega))} \quad \forall \v \in L^s(I,L^2(\Omega)^d), \; 1 \le s \le \infty
\end{equation}
and
\begin{equation}\label{eq:P_tau_estimate}
\norm{\v - P_\tau \v}_{L^\infty(I,L^2(\Omega))} \le C \tau^{1-\frac{1}{s}}\norm{\partial_t \v}_{L^s(I,L^2(\Omega))} \quad \forall \v \in W^{1,s}(I,L^2(\Omega)^d), \; 1 \le s \le \infty.
\end{equation}
We define the bilinear form $\mathfrak{B}$ by
    \[
	\mathfrak{B}(\u,\v) = \sum_{m=1}^M (\partial_t\u, \v )_{I_m\times \Omega} + (\nabla \u, \nabla \v)_{I\times \Omega}
	+ \sum_{m=2}^M ([\u]_{m-1},\v^+_{m-1})_{\Omega} + (\u_0^+,\v_0^+)_{\Omega}.
    \]
    With this bilinear form we define the fully discrete approximation for the transient Stokes problem on the discrete divergence free space $X_{\tau}^w(\V_h)$:
    \begin{equation}\label{eq:fully_discrete_div_free}
	\u_{\tau h} \in X_{\tau}^w(\V_h) \;:\; \mathfrak{B}(\u_{\tau h},\v_{\tau h}) = (\f,\v_{\tau h})_{I \times \Omega} + (\u_0, \v_{\tau h,0}^+)_{\Omega} \quad \forall \v_{\tau h} \in X_{\tau}^w(\V_h).
    \end{equation}
    By a standard argument one can see that this formulation possesses a unique solution (existence follows from uniqueness by the fact that \cref{{eq:fully_discrete_div_free}} is equivalent to a quadratic system of linear equations). 
\begin{remark}\label{remark_on_Ph}
    Note, that the data $\f$ and $\u_0$ in \cref{eq:fully_discrete_div_free} can be replaced by $\mathbb{P}_h \f$ and $\mathbb{P}_h \u_0$ respectively (with $\mathbb{P}_h$ being the discrete Leray projection \cref{eq:discrete_Leray_projection})  without changing the solution. Therefore, this formulation makes sense for a general $\f \in L^1(I; L^1(\Omega)^d)$ and $\u_0 \in L^1(\Omega)^d$. However, for the error analysis later on we will require the assumptions from \cref{theorem:weak_with_preasure}, ensuring the Galerkin orthogonality relation, see also \cref{prop:Galerkin_orthogonality} below.
\end{remark}
The above formulation is not a conforming discretization of the divergence free formulation \cref{eq: weak stokes} due to the fact that $X_{\tau}^w(\V_h)$ is not a subspace of $L^2(I;\V^1)$. In order to introduce a velocity pressure discrete formulation (as a discretization of \cref{eq: weak stokes with pressure}) we consider the following bilinear form
\begin{multline}
    B((\u,p),(\v,q)) = \sum_{m=1}^M ( \partial_t\u, \v )_{I_m\times \Omega} + (\nabla \u, \nabla \v)_{I\times \Omega} - (p,\nabla \cdot \v)_{I\times\Omega} + (\nabla \cdot \u,q)_{I\times\Omega}\\
    + \sum_{m=2}^M ([\u]_{m-1},\v^+_{m-1})_{\Omega} + (\u_0^+,\v_0^+)_{\Omega}.
\end{multline}
The corresponding fully discrete formulation reads: find $(\u_{{\tau}h},p_{{\tau}h})\in X^w_{\tau}(\X_h \times M_h)$ such that
\begin{equation}\label{eq:spacetime_discretization}
   B((\u_{{\tau}h},p_{{\tau}h}),(\v_{{\tau}h},q_{{\tau}h})) = (\f,\v_{{\tau}h})_{I\times \Omega} + (\u_0, \v_{{\tau}h,0}^+)_{\Omega} \qquad \forall (\v_{{\tau}h},q_{{\tau}h}) \in X^w_{\tau}(\X_h \times M_h).
\end{equation}
We note, that for the temporal discretization we use polynomials of the same order for the velocity and pressure. The next proposition states the equivalence of the formulation \cref{eq:fully_discrete_div_free} and \cref{eq:spacetime_discretization}.
\begin{proposition}
For a solution $(\u_{{\tau}h},p_{{\tau}h})$ of \cref{eq:spacetime_discretization} the discrete velocity $\u_{{\tau}h}$ fulfills \cref{eq:fully_discrete_div_free}. Moreover, for a solution $\u_{{\tau}h}$ of \cref{eq:fully_discrete_div_free} there exists a unique $p_{{\tau}h} \in X^w_{\tau}(M_h)$ such that the pair $(\u_{{\tau}h},p_{{\tau}h})$ fulfills \cref{eq:spacetime_discretization}. In particular the solution of \cref{eq:spacetime_discretization} is unique.
\end{proposition}
\begin{proof}
From \cref{eq:spacetime_discretization} we have $\u_{{\tau}h} \in X^w_{\tau}(\V_h)$ and so it trivially fulfills \cref{eq:fully_discrete_div_free}. Let now  $\u_{{\tau}h} \in X^w_{\tau}(\V_h)$ be the solution of \cref{eq:fully_discrete_div_free}. We define $\w_{\tau h} \in X^w_{\tau}(\X_h)$ by
\[
(\w_{\tau h},\v_{{\tau}h})_{I \times \Omega} =  (\f,\v_{\tau h})_{I \times \Omega} + (\u_0, \v_{\tau h,0}^+)_{\Omega} - \mathfrak{B}(\u_{\tau h},\v_{\tau h}) \quad \forall \v_{\tau h} \in X_{\tau}^w(\X_h).
\]
It follows immediately
\[
(\w_{\tau h},\v_{{\tau}h})_{I \times \Omega} = 0 \quad \forall \v_{\tau h} \in X_{\tau}^w(\V_h)
\]
and one obtains $\w_{\tau h}(t) \in \V_h^\perp$ for every $t \in I_m$, $m=1,2,\dots, M$. At the same time we have globally $\w_{\tau h} \in X_{\tau}^w(\V_h^\perp)$. By \cref{lemma:stationary_existence_pressure} we get the existence and uniqueness of the pressure $p_{\tau h}(t)\in M_h$ with
\[
(\w_{\tau h}(t),\v_h) = (p_{\tau h}(t),\nabla \cdot \v_h) \quad \forall \v_h \in \X_h.
\] 
Therefore $p_{\tau h} \in X_{\tau}^w(M_h)$ and there holds
\[
(\w_{\tau h},\v_{{\tau}h})_{I \times \Omega} = (p_{\tau h},\nabla \cdot \v_{\tau h})_{I \times \Omega} \quad \forall \v_{\tau h} \in X_{\tau}^w(\X_h).
\]
This completes the proof.
\end{proof}

The next proposition provides the Galerkin orthogonality relation for the velocity pressure discretization \cref{eq:spacetime_discretization}, which is essential for our analysis. Please note, that for the velocity formulation \cref{eq:fully_discrete_div_free} the Galerkin orthogonality does not hold due to the fact that $X_{\tau}^w(\V_h)$ is not a subspace of $L^2(I,\V^1)$.
\begin{proposition}\label{prop:Galerkin_orthogonality}
Let the assumptions of \cref{theorem:weak_with_preasure} be fulfilled, i.e., $\f \in L^s(I;L^2(\Omega)^d)$ for some $s>1$ and $\u_0 \in \V^0_{1-\frac{1}{s}}$. Then there holds for the solution $(\u,p)$ of \cref{eq: weak stokes with pressure}
\[
B((\u,p),(\v_{\tau h},q_{\tau h})) =  (\f,\v_{{\tau}h})_{I\times \Omega} + (\u_0, \v_{{\tau}h,0}^+)_{\Omega} \qquad \forall (\v_{{\tau}h},q_{{\tau}h}) \in X^w_{\tau}(\X_h \times M_h)
\]
and consequently 
\begin{equation}\label{eq:Galerkin_orthogonality}
B((\u-\u_{\tau h},p-p_{\tau h}),(\v_{\tau h},q_{\tau h})) = 0 \qquad \forall (\v_{{\tau}h},q_{{\tau}h}) \in X^w_{\tau}(\X_h \times M_h).
\end{equation}
\end{proposition} 
\begin{proof}
In the setting of \cref{theorem:weak_with_preasure} we have $\u \in L^2(I,\V^1) \cap C(\bar I,\V^0)$, $\partial_t \u \in L^s(I,L^2(\Omega)^d)$ and $p \in L^s(I,L^2(\Omega))$.  Therefore all terms in the bilinear form are well defined. For the test space we have
\[
X^w_{\tau}(\X_h) \subset  L^2(I;H^1_0(\Omega)^d) \cap L^\infty(I;L^2(\Omega)^d)\quad \text{and}\quad  X^w_{\tau}(M_h) \subset L^2(I;L^2_0(\Omega)).
\]
Therefore, we can choose $(\v_{{\tau}h},q_{{\tau}h}) \in X^w_{\tau}(\X_h \times M_h)$ as test functions in \cref{eq: weak stokes with pressure}. Moreover, all jump terms vanish due to $\u \in C(\bar I,\V^0)$ and 
\[
(\u_0^+,\v_{\tau h,0}^+) = (\u_0,\v_{\tau h,0}^+)
\]
due to $\u(0) = \u_0$. This completes the proof.
\end{proof}

In the following, we also consider a dual problem, where we use a dual representation of the bilinear form $B$
\begin{equation}\label{eq:dualB}
\begin{multlined}
    B((\u,p),(\v,q)) = -\sum_{m=1}^M \langle \u, \partial_t\v \rangle_{I_m\times \Omega} + (\nabla \u, \nabla \v)_{I\times \Omega} - (p,\nabla \cdot \v)_{I\times\Omega}\\ + (\nabla \cdot \u,q)_{I\times\Omega}
    - \sum_{m=1}^{M-1} (\u^{-}_{m},[\v]_{m})_{\Omega} + (\u_M^-,\v_M^-)_{\Omega},
\end{multlined}
\end{equation}
which is obtained by integration by parts and rearranging the terms in the sum.

\section{Fully discrete smoothing and maximal regularity estimates}\label{sec:maximal_regularity}
The goal of this section is to extend the results on the discrete maximal parabolic regularity  for the discretization of the heat equation from \cite{2017Leykekhman} to the transient Stokes equations.
The results in \cite{2017Leykekhman} rely solely on the resolvent estimates for $-\Delta_h$ and the Dunford-Taylor operator calculus. Since \cref{lemma:l2_resolvent_discrete} establishes the resolvent estimate for $A_h$, all the results from \cite{2017Leykekhman} continue to hold for $A_h$ as well.  We will state the results below.

The first result is a smoothing estimate for the homogeneous problem ($f=0$). 

\begin{theorem}\label{chap:IS:theorem:maximal_regularity_smoothing_discrete}
    Let $\f=\vec{0}$ and let $\u_0 \in L^2(\Omega)^d$. Let $\u_{{\tau}h}\in X^w_{\tau}(\V_h)$ be the solution to \cref{eq:fully_discrete_div_free}. Then, there holds for $m= 1,2, \dots, M$
    \begin{equation}\label{eq:smoothing_estimate}
	\norm{\partial_t  \u_{{\tau}h}}_{L^{\infty}(I_m;L^2(\Omega))} + \norm{A_h \u_{{\tau}h}}_{L^{\infty}(I_m;L^2(\Omega))}
	+  \norm{{\tau}^{-1}_m[\u_{{\tau}h}]_{m-1}}_{L^2(\Omega)} \leq \frac{C}{t_m} \norm{\mathbb{P}_h\u_0}_{L^2(\Omega)}.
    \end{equation}
    Here we have $[\u_{\tau h}]_0 = \u^+_{\tau h,0} - \mathbb{P}_h\u_0$.
\end{theorem}
\begin{proof}
The key step in proving this smoothing estimate, is the representation of the  solution on $I_m$ in form of the Dunford-Taylor integral, (cf.  \cite[p. 1321--1322]{1998Eriksson})
\begin{equation}
A_h \u_{\tau,m}^-= \frac{1}{2\pi i}\int_{\Gamma}\prod_{l=1}^m r(\tau_l z) A_h R(z, A_h)dz \mathbb{P}_h\u_0 \qquad \text{for }\;m=2,\dots,M,
\end{equation}
which is an operator equality on $\V_h$ since $\u_0$ is replaced by $\mathbb{P}_h\u_0 \in \V_h$, cf. \cref{remark_on_Ph}. Here, $r(z)$ is a subdiagonal Pad\`{e} approximation, which is rational function with the numerator of degree $w$ and denominator of degree $w+1$. The contour $\Gamma$ is a curve contained in the resolvent set of $A_h$ such that \cref{lemma:l2_resolvent_discrete} can be applied for $-z \in \Gamma$ and $R(z, A_h)$ the resolvent operator, i.e., $R(z,A_h) = (z - A_h)^{-1}$. Then the proof of
 \[
\norm{A_h \u_{{\tau}h}}_{L^{\infty}(I_m;L^2(\Omega))} \le  \frac{C}{t_m} \norm{\mathbb{P}_h\u_0}_{L^2(\Omega)}
 \]
 is the same as for \cite[Theorem 5.1]{1998Eriksson}. The estimates for the time derivative and for the jumps follow as in \cite[Theorem 4 and Theorem 5]{2017Leykekhman}.
\end{proof}

For the inhomogeneous problem (and $\u_0=0$) we obtain the discrete analog of \cref{proposition:max_reg}. On the continuous level the corresponding estimate is true for $1 < s < \infty$. The following discrete maximal parabolic regularity result covers the limit cases $s=1$ and $s=\infty$ at the expense of an logarithmic term.  
\begin{theorem}\label{chap:IS:corollary:maximal_regularity_discrete}
    Let $1 \leq s \le  \infty$, $\f \in L^s(I,L^2(\Omega)^d)$ and $\u_0 =0$. Let $\u_{{\tau}h}\in X^w_{\tau}(\V_h)$ be the solution to \cref{eq:fully_discrete_div_free}. Then for $s < \infty$ there holds
    \begin{multline}\label{eq:dmpr}
	\Bigg( \sum_{m=1}^M \norm{\partial_t \u_{{\tau}h}}^s_{L^s(I_m; L^2(\Omega))}\Bigg)^{1/s} + \norm{A_h \u_{{\tau}h}}_{L^s(I;L^2(\Omega))}\\%
	+ \Bigg( \sum^M_{m=1} {\tau}_m \norm{{\tau}_m^{-1} [\u_{{\tau}h}]_{m-1}}^s_{L^2(\Omega)} \Bigg)^{1/s} 
	\leq C \ln \frac{\finaltime}{{\tau}} \norm{\mathbb{P}_h\f}_{L^s(I; L^2(\Omega))}.
    \end{multline}
    For $s=\infty$ the estimate takes the form 
    \begin{equation}\label{eq:dmpr_infty}
	\max_{1\leq m \leq M} \norm{\partial_t \u_{{\tau}h}}_{L^\infty(I_m; L^2(\Omega))} + \norm{A_h \u_{{\tau}h}}_{L^\infty(I;L^2(\Omega))}%
	+ \max_{1\leq m \leq M} \norm{{\tau}_m^{-1} [\u_{{\tau}h}]_{m-1}}_{L^2(\Omega)}
	\leq C \ln \frac{\finaltime}{{\tau}} \norm{\mathbb{P}_h\f}_{L^\infty(I; L^2(\Omega))}.
    \end{equation}
    Here we have $[\u_{\tau h}]_0 = \u^+_{\tau h,0}$.
\end{theorem}
\begin{proof}
The result follows from the smoothing estimate in \cref{chap:IS:theorem:maximal_regularity_smoothing_discrete} as in the proof of \cite[Theorems 6-8]{2017Leykekhman}.
\end{proof}
\begin{remark}
Due to the stability of the discrete Leray projection $\mathbb{P}_h$ in $L^2$, we can drop it in the above estimates.
\end{remark}

\begin{remark}\label{remark:delta_est}
If we assume the domain $\Omega$ to be convex and the family of meshes $\{\Th\}$ to be shape regular and quasi-uniform then there holds
\[
\norm{\Delta_h \v_h}_{L^2(\Omega)} \le c \norm{A_h \v_h}_{L^2(\Omega)} \quad \forall \v_h \in \V_h
\]
by \cite[Lemma 4.1]{2008Guermond} or \cite[Corollary 4.4]{1982Heywood} and  therefore the corresponding estimates hold also for $\norm{\Delta_h \u_{{\tau}h}}_{L^s(I;L^2(\Omega))}$.
\end{remark}

For $s=2$ it is possible to get rid of the logarithmic term in \cref{eq:dmpr} similar to \cite[Theorem 4.6]{MeidnerVexler:2008I}.
\begin{theorem}
For $\f \in L^2(I,L^2(\Omega)^d)$ and $\u_0 =0$. Let $\u_{{\tau}h}\in X^w_{\tau}(\V_h)$ be the solution to \cref{eq:fully_discrete_div_free}. Then  there holds
\begin{multline}\label{eq:dmpr_s2}
	\Bigg( \sum_{m=1}^M \norm{\partial_t \u_{{\tau}h}}^2_{L^2(I_m; L^2(\Omega))}\Bigg)^{1/2} + \norm{A_h \u_{{\tau}h}}_{L^2(I;L^2(\Omega))}\\%
	+ \Bigg( \sum^M_{m=1} {\tau}_m \norm{{\tau}_m^{-1} [\u_{{\tau}h}]_{m-1}}^2_{L^2(\Omega)} \Bigg)^{1/2} 
	\leq C  \norm{\mathbb{P}_h\f}_{L^2(I; L^2(\Omega))}.
\end{multline}
\end{theorem}
\begin{proof}
The proof follows the lines of the corresponding proof in \cite{MeidnerVexler:2008I} replacing $-\Delta_h$ by $A_h$.	
\end{proof}	
\section{Best approximation type estimates}\label{sec:bestapproximation}

The results in \cref{sec:maximal_regularity} allow us to show an $L^{\infty}(I; L^2(\Omega))$ best approximation type  error estimates for the velocity. 
In order to state the results, we need to introduce an analogue of the Ritz projection for the stationary Stokes problem  $(R_h^S(\w,\varphi), R_h^{S,p}(\w,\varphi))\in \X_h \times M_h$ of $(\w,\varphi)\in H^1_0(\Omega)^d\times L^2(\Omega)$ given by the relation \sbox0{\cref{chap:IS:eq:stokes_projection_1,chap:IS:eq:stokes_projection_2}}
\begin{subequations}\label{chap:IS:eq:stokes_projection}
\begin{align}
    (\nabla (\w - R_h^S(\w,\varphi)),\nabla \v_h) - (\varphi-R_h^{S,p}(\w,\varphi), \nabla \cdot \v_h) &= 0, \qquad \forall \v_h \in \X_h,\label{chap:IS:eq:stokes_projection_1}\\
    (\nabla \cdot (\w - R_h^S(\w,\varphi)), q_h) &=0, \qquad \forall q_h \in M_h.\label{chap:IS:eq:stokes_projection_2}
\end{align}
\end{subequations}
\begin{remark}\label{remark:RitzStokesProjection}
If $\w$ is discrete divergence free, i.e., $(\nabla\cdot\w,q_h)=0$ for all $q_h\in M_h$, then we have $R_h^S(\w,\varphi) \in \V_h$. We will use this projection operator only for such $\w$. In this case the same projection operator is defined, e.g., in~\cite{Girault:2015}.
\end{remark}
In the following, we will make the same assumptions on the data $\f$ and $\u_0$  as in \cref{theorem:weak_with_preasure} in order to use the Galerkin orthogonality relation from \cref{prop:Galerkin_orthogonality}.
\begin{theorem}\label{chap:IS:theorem:L2Linfty_stability}
	Let $\f \in L^s(I;L^2(\Omega)^d)$ for some $s>1$ and $\u_0 \in \V^0_{1-\frac{1}{s}}$.
    Let $(\u,p)$ be the solution of \cref{eq: weak stokes with pressure} and $(\u_{{\tau}h},p_{\tau h})$ solve the respective finite element problem \cref{eq:spacetime_discretization}. Then, there holds
    \begin{equation}
	\norm{\u_{{\tau}h}}_{L^{\infty}(I;L^2(\Omega))} \leq C \ln \frac{\finaltime}{{\tau}} \Big(\norm{\u}_{L^{\infty}(I;L^2(\Omega))} + \norm{\u - R_h^S \u}_{L^{\infty}(I;L^2(\Omega))}\Big).
    \end{equation}
\end{theorem}
\begin{proof}
	We proceed with a proof along the arguments of \cite[Theorem 1]{2016Leykekhmana}.
    Let $\tilde t \in (0,\finaltime]$ and without loss of generality assume $\tilde t \in (t_{M-1},\finaltime]$. %

    Consider the following dual  problem
    \begin{subequations}
	\begin{alignat}{3}
	    -\partial_t \g(t,\x) - \Delta \g(t,\x) + \nabla \lambda(t,\x) &=\u_{{\tau}h}(\tilde t,\x)\theta(t) , \qquad && (t,\x) \in I\times \Omega,\\
	    \nabla \cdot \g(t,\x) &= 0, \qquad && (t,\x) \in I\times \Omega,\\
	    \g(t,\x) &= 0, \qquad &&(t,\x) \in I\times \partial\Omega,\\
	    \g(\finaltime,x) &= 0, \qquad && \x \in \Omega.
	\end{alignat}
    \end{subequations}
    Here, $\theta \in C^1(I)$ is a regularized Delta function (cf. \cite[Appendix A.5]{1995Schatz}) in time with the following properties:
    \begin{equation}\label{theta_def}
	\supp \theta \subset (t_{M-1},\finaltime), \qquad \norm{\theta}_{L^1(I_M)} \leq C \qquad\text{ and }\qquad (\theta, \v_{\tau})_{I_M} = \v_{\tau}(\tilde t) \quad \forall \v_{\tau} \in \mathcal{P}_w(I_M). \label{eq:delta_t}
    \end{equation}
    Note that the authors in \cite[Appendix A.5]{1995Schatz} assume $\tilde t$ to be an element of an open interval but the argument there can be extended to the case $\tilde t = \finaltime$.
    The corresponding finite element approximation $(\g_{{\tau}h}, \lambda_{{\tau}h}) \in X^w_{\tau}(\X_h\times M_h)$ is given by
    \begin{equation}
    B((\v_{{\tau}h},q_{{\tau}h}),(\g_{{\tau}h},\lambda_{{\tau}h})) = \left(\u_{\tau h}(\tilde t) \theta,\v_{{\tau}h}\right)_{I \times \Omega} \qquad \forall(\v_{{\tau}h},q_{{\tau}h})\in X^w_{\tau}(\X_h\times M_h).
    \end{equation}
By the Galerkin orthogonality from \cref{prop:Galerkin_orthogonality} we have
    \begin{align}
	\stwonorm{\u_{{\tau}h}(\tilde t)} 
	&= (\u_{{\tau}h}, \theta(t) \u_{{\tau}h}(\tilde t))
			      = B((\u_{{\tau}h},p_{{\tau}h}),(\g_{{\tau}h},\lambda_{{\tau}h})) = B((\u,p),(\g_{{\tau}h},\lambda_{{\tau}h})) \\
			      &= -\sum_{m=1}^M(\u,\partial_t\g_{{\tau}h})_{I_m \times \Omega} + (\nabla \u, \nabla \g_{{\tau}h})_{I\times \Omega} - (p, \nabla \cdot \g_{{\tau}h}) - \sum_{m=1}^M(\u_m^-,[\g_{{\tau}h}]_m)_{\Omega}\\
			      &= J_1 + J_2 + J_3 + J_4,
    \end{align}
    where we have used the dual representation of the bilinear form $B$ from \cref{eq:dualB}.
    In the last sum we set $\g_{{\tau}h,M+1}=0$ so that $[\g_{{\tau}h}]_M= - \g_{{\tau}h,M}$.   
    Applying the H\"older's inequality and using $\u \in C(\bar I;L^2(\Omega)^d)$, we obtain   
    \begin{align}
	J_1 &\leq \sum_{m=1}^M \norm{\u}_{L^{\infty}(I_m;L^2(\Omega))} \norm{\partial_t \g_{{\tau}h}}_{L^1(I_m;L^2(\Omega))}
	    \leq \norm{\u}_{L^{\infty}(I;L^2(\Omega))}\sum_{m=1}^M  \norm{\partial_t \g_{{\tau}h}}_{L^1(I_m;L^2(\Omega))}, \\
	J_4 &\leq \sum_{m=1}^{M} \twonorm{\u_m^-}\twonorm{[\g_{{\tau}h}]_{m-1}} \leq \norm{\u}_{L^{\infty}(I;L^2(\Omega))} \sum_{m=1}^{M} \twonorm{[\g_{{\tau}h}]_{m}}.
    \end{align} 
    For $J_2+J_3$ we can argue by using the projection $R_h^S$ defined in \cref{chap:IS:eq:stokes_projection}.
    Then, we have
    \begin{align}
	J_2 + J_3 &=(\nabla \u, \nabla \g_{{\tau}h})_{I\times \Omega} - (p,\nabla \cdot \g_{{\tau}h})_{I\times\Omega}\\
		  &=(\nabla R_h^S(\u,p), \nabla \g_{{\tau}h})_{I\times \Omega} - (R_h^{S,p}(\u,p),\nabla \cdot \g_{{\tau}h})_{I\times\Omega}
		  =(\nabla R_h^S(\u,p), \nabla \g_{{\tau}h})_{I\times \Omega},
    \end{align}
    where the last term vanishes, since $\g_{{\tau}h}$ is discretely divergence-free. Here and in what follows, the projection $(R_h^S,R_h^{S,p})$ is applied to time dependent functions $(\u,p)$ pointwise in time. Since $\nabla \cdot \u(t) = 0$ for almost all $t \in I$ we have $R_h^S(\u(t),p(t)) \in \V_h$, cf. \cref{remark:RitzStokesProjection}.
    With this we can use the definition of the discrete Stokes operator $A_h$ resulting in 
    \begin{equation}\label{eq:est_R_h_g_h}
    \begin{aligned}
	(\nabla R_h^S(\u,p), \nabla \g_{{\tau}h}&)_{I\times \Omega} 
	= (R_h^S(\u,p), A_h \g_{{\tau}h})_{I\times \Omega}\\
	\leq& \left(\norm{\u}_{L^{\infty}(I;L^2(\Omega))} + \norm{\u - R_h^S (\u,p)}_{L^{\infty}(I;L^2(\Omega))}\right) \norm{A_h\g_{{\tau}h}}_{L^1(I;L^2(\Omega))}.
    \end{aligned}
    \end{equation}
    Combining the estimates, we conclude    
    \begin{align}
	\stwonorm{\u_{{\tau}h}(\tilde t)} 
	&=-\sum_{m=1}^M(\u,\partial_t\g_{{\tau}h})_{I_m \times \Omega} + (\nabla R^S_h(\u,p), \nabla \g_{{\tau}h})_{I\times \Omega} - \sum_{m=1}^M(\u_m^-,[\g_{{\tau}h}]_m)_{\Omega} \\
	&\leq C \Big(\norm{\u}_{L^{\infty}(I;L^2(\Omega))} +  \norm{\u - R_h^S(\u,p)}_{L^{\infty}(I;L^2(\Omega))}\Big) \\
	&\qquad \times\Bigg( \sum_{m=1}^M\norm{\partial_t\g_{{\tau}h}}_{L^1(I_m;L^2(\Omega))} + \norm{A_h \g_{{\tau}h}}_{L^1(I;L^2(\Omega))} 
	 + \sum_{m=1}^M \norm{[\g_{{\tau}h}]_m}_{L^2(\Omega)}\Bigg)
    \end{align}
    and an application of \cref{chap:IS:corollary:maximal_regularity_discrete} leads to
    \begin{align}
    	\stwonorm{\u_{{\tau}h}(\tilde t)}
	\leq C \ln \frac{\finaltime}{{\tau}} \Big(\norm{\u}_{L^{\infty}(I;L^2(\Omega))} +  \norm{\u - R_h^S(\u,p)}_{L^{\infty}(I;L^2(\Omega))}\Big)
	 \times\twonorm{\u_{{\tau}h}(\tilde t)}\norm{\theta}_{L^1(I_M)}.
    \end{align}
    Canceling and using that $\norm{\theta}_{L^1(I_M)}\le C$ we complete the proof of the theorem.
\end{proof}

The following corollary provides a version of \cref{chap:IS:theorem:L2Linfty_stability} involving the $L^2$ projection in time $P_\tau$ defined in \cref{eq:P_tau}.
\begin{corollary}\label{corollary:stability_with_P_tau}
Under the conditions of \cref{chap:IS:theorem:L2Linfty_stability} there holds
\[
\norm{\u_{{\tau}h}}_{L^{\infty}(I;L^2(\Omega))} \leq C \ln \frac{\finaltime}{{\tau}} \Big(\norm{\u}_{L^{\infty}(I;L^2(\Omega))} + \norm{\u - P_\tau R_h^S \u}_{L^{\infty}(I;L^2(\Omega))}\Big).
\]
\end{corollary}
\begin{proof}
We obtain this by arguing as in the proof of \cref{chap:IS:theorem:L2Linfty_stability}, only the term in \cref{eq:est_R_h_g_h} is estimated differently
\[
\begin{aligned}
(\nabla R_h^S(\u,p), \nabla \g_{{\tau}h})_{I\times \Omega} 
&= (R_h^S(\u,p), A_h \g_{{\tau}h})_{I\times \Omega}
= (P_\tau R_h^S(\u,p), A_h \g_{{\tau}h})_{I\times \Omega}\\
&\leq \Big(\norm{\u}_{L^{\infty}(I;L^2(\Omega))} + \norm{\u - P_\tau R_h^S (\u,p)}_{L^{\infty}(I;L^2(\Omega))}\Big)
\norm{A_h\g_{{\tau}h}}_{L^1(I;L^2(\Omega))},
\end{aligned}
\]
where we have used the definition \cref{eq:P_tau} of $\P_\tau$.
\end{proof}

As a corollary from \cref{chap:IS:theorem:L2Linfty_stability} we obtain a best approximation type result.
\begin{corollary}
    \label{chap:IS:corr_best_approx}
    Under the conditions of \cref{chap:IS:theorem:L2Linfty_stability} there holds
    \begin{align}
	\norm{\u-\u_{{\tau}h}}_{L^{\infty}(I;L^2(\Omega))} 
	\leq C \ln \frac{\finaltime}{{\tau}} \Bigg(\inf_{\v_{\tau h} \in X^w_\tau(\V_h)}\norm{\u-\v_{{\tau}h}}_{L^{\infty}(I;L^2(\Omega))} +  \norm{\u - R_h^S (\u,p)}_{L^{\infty}(I;L^2(\Omega))} \Bigg). \label{chap:IS:eq:convergence}
    \end{align}
\end{corollary}
\begin{proof}
The desired result follows by
    considering $(\u -\v_{{\tau}h},p-q_{{\tau}h})$ instead of $(\u,p)$ with arbitrary $(\v_{{\tau}h},q_{{\tau}h}) \in X^w_{\tau}(\V_h \times M_h)$ in the proof of \cref{chap:IS:theorem:L2Linfty_stability}. This allows us to replace $(\u_{\tau h}, p_{\tau h})$
by $(\u_{\tau h} -\v_{{\tau}h},p_{\tau h}-q_{{\tau}h})$.  

Note, that $\u-\v_{\tau h}$ is discrete divergence free and so $R_h^S(\u-\v_{\tau h},p- q_{\tau h})(t)\in \V_h$ for almost all $t \in I$, see \cref{remark:RitzStokesProjection}. Therefore, the argument in the proof of \cref{chap:IS:theorem:L2Linfty_stability} involving the discrete Stokes operator $A_h$ is still valid. Moreover, by the definition of the projection $R_h^S$ we have
    \[
    R_h^S(\u-\v_{\tau h},p- q_{\tau h}) = R_h^S(\u,p) - \v_{\tau h}.
    \]
    As in the proof of \cref{chap:IS:theorem:L2Linfty_stability} we obtain then
    \[
    \norm{\u_{{\tau}h}-\v_{{\tau}h}}_{L^{\infty}(I;L^2(\Omega))} \le C \ln \frac{\finaltime}{\tau} \left( \norm{\u-\v_{{\tau}h}}_{L^{\infty}(I;L^2(\Omega))} 
     + \norm{\u - R_h^S(\u,p)}_{L^{\infty}(I;L^2(\Omega))}\right).
    \]
    We complete the proof using
    \[
    \u-\u_{{\tau}h} = \u-\v_{{\tau}h} + \v_{{\tau}h} -\u_{{\tau}h}
    \]
    and the triangle inequality.
\end{proof}

The error estimate in the next section is based on the following variant of our result.
\begin{corollary}\label{corollary:best_app_P_tau_R_h}
Under the conditions of \cref{chap:IS:theorem:L2Linfty_stability} there holds
\[
\norm{\u-\u_{{\tau}h}}_{L^{\infty}(I;L^2(\Omega))} \le C \ln \frac{\finaltime}{{\tau}} \norm{\u - P_\tau R_h^S (\u,p)}_{L^{\infty}(I;L^2(\Omega))}.
\]
\end{corollary}
\begin{proof}
Using \cref{corollary:stability_with_P_tau} and arguing as in the proof of \cref{chap:IS:corr_best_approx} we obtain
\[
\norm{\u-\u_{{\tau}h}}_{L^{\infty}(I;L^2(\Omega))} 
\leq C \ln \frac{\finaltime}{{\tau}} \Bigg(\inf_{\v_{\tau h} \in X^w_\tau(\V_h)}\norm{\u-\v_{{\tau}h}}_{L^{\infty}(I;L^2(\Omega))} +  \norm{\u - P_\tau R_h^S (\u,p)}_{L^{\infty}(I;L^2(\Omega))} \Bigg). 
\]
Choosing $\v_{{\tau}h} = P_\tau R_h^S (\u,p) \in X^w_{\tau}(\V_h)$ we obtain the result.
\end{proof}

\section{Error estimates and comparison to the literature}\label{sec:error_estimates}
In this section we apply our result from \cref{corollary:best_app_P_tau_R_h} to derive a priori error estimates. Due to the nature of the result from \cref{corollary:best_app_P_tau_R_h} we obtain error estimates which are optimal (probably up to a logarithmic term) with respect to both the orders of approximation and the assumed regularity. Moreover, we compare these estimates with the results from the literature. For this section we assume the domain $\Omega$ to be polygonal/polyhedral and convex. 
\begin{remark}\label{remark:non_convex}
Note that the results from \cref{chap:IS:corr_best_approx} or \cref{corollary:best_app_P_tau_R_h} can be applied also to non-convex domains and to meshes, which are not necessarily quasi-uniform, including graded or even anisotropic refinement towards reentrant corners or edges. In this case one can use error estimates for the solution of the stationary Stokes equations for such cases, see, \cite{Apel:2021} and the references therein, in order to estimate $\u - R_h^S(\u,p)$ in  \cref{chap:IS:eq:convergence}.
\end{remark}

For this section we assume the following standard approximation properties for the spaces $\X_h$ and $M_h$.
\begin{assumption}\label{assumption_interpolation}
There exists an interpolation operator $i_h \colon H^2(\Omega)^d \cap H^1_0(\Omega)^d \to \X_h$ and $r_h \colon L^2(\Omega) \to M_h$ such that
\[
\norm{\nabla(\v-i_h v)}_{L^2(\Omega)} \le c h \norm{\nabla^2 v} _{L^2(\Omega)} \quad \forall \v \in H^2(\Omega)^d \cap H^1_0(\Omega)
\]
and
\[
\norm{q-r_h q}_{L^2(\Omega)} \le c h \norm{\nabla q} _{L^2(\Omega)} \quad \forall q \in H^1(\Omega).
\]	
\end{assumption}
This assumption is fulfilled for a variety of finite element pairs including, e.g., Taylor-Hood as well as mini element on a family of shape regular meshes.
 Under this assumption the following standard estimate holds for the Ritz projection for the Stokes problem \cref{chap:IS:eq:stokes_projection}.

\begin{proposition}\label{prop:estimate_Ritz_Stokes}
Let $\Omega$ be convex and \cref{assumption_interpolation} be fulfilled. There is a constant $C>0$ such that for all $(\u,p)$ with $\u \in H^2(\Omega)^d\cap \V^1$ and $p \in H^1(\Omega) \cap L^2_0(\Omega)$ the following estimate holds
\[
\norm{\u-R_h^S(\u,p)}_{L^2(\Omega)} \le C h^2 \left(\norm{\nabla^2 \u}_{L^2(\Omega)} + \norm{\nabla p}_{L^2(\Omega)}\right).
\]
\end{proposition}
\begin{proof}
We refer, e.g., to \cite[Theorem 1.9]{1986Girault}.
\end{proof}

In the following theorem we provide an error estimate of order ${\mathcal O}(\tau +  h^2)$ up to a logarithmic term under minimal assumption on the data. The estimate holds for every choice of degree $w$ in the temporal discretization but especially for $w=0$, i.e., for the dG($0$) discretization, which is known to be a variant of the implicit Euler scheme.

\begin{theorem}\label{theorem:error_estimate}
Let $\Omega$ be convex, $\f \in L^\infty(I,L^2(\Omega)^d)$ and $\u_0 \in \V^2$ and let \cref{assumption_interpolation} be fulfilled. Let $(\u,p)$ be the solution of \cref{eq: weak stokes with pressure} and $(\u_{{\tau}h},p_{\tau h})$ solve the respective finite element problem \cref{eq:spacetime_discretization}. Then, there holds
    \begin{equation}
\norm{\u-\u_{{\tau}h}}_{L^{\infty}(I;L^2(\Omega))}\nonumber
\leq C \left(\ln \frac{\finaltime}{{\tau}}\right)^2\left({\tau} + h^2\right) \left( \norm{\f}_{L^{\infty}(I;L^2(\Omega))} + \norm{\u_0}_{\V^2}\right).
\end{equation}
\end{theorem}
\begin{proof}
	We start with the result from \cref{corollary:best_app_P_tau_R_h} and obtain
	\[
	\begin{aligned}
	    \norm{\u-\u_{{\tau}h}}_{L^{\infty}(I;L^2(\Omega))} 
	&\le C \ln \frac{\finaltime}{{\tau}} \norm{\u - P_\tau R_h^S (\u,p)}_{L^{\infty}(I;L^2(\Omega))}\\
	&\le C\ln \frac{\finaltime}{{\tau}} \left(\norm{\u - P_\tau \u}_{L^{\infty}(I;L^2(\Omega))}+\norm{P_\tau (\u - R_h^S (\u,p))}_{L^{\infty}(I;L^2(\Omega))}\right)\\
	&\le C\ln \frac{\finaltime}{{\tau}} \left(\norm{\u - P_\tau \u}_{L^{\infty}(I;L^2(\Omega))}+\tau^{-\frac{1}{s}}\norm{P_\tau (\u - R_h^S (\u,p))}_{L^s(I;L^2(\Omega))}\right)\\
	&\le C\ln \frac{\finaltime}{{\tau}} \left(\norm{\u - P_\tau \u}_{L^{\infty}(I;L^2(\Omega))}+\tau^{-\frac{1}{s}}\norm{\u - R_h^S (\u,p)}_{L^s(I;L^2(\Omega))}\right),
	\end{aligned}
	\]
	where we have used an inverse inequality for some $1<s<\infty$ and the stability of $P_\tau$ in $L^s$ from \cref{eq:P_tau_stability}. The temporal projection error is estimated by \cref{eq:P_tau_estimate} resulting in
	\[
	\norm{\u - P_\tau \u}_{L^{\infty}(I;L^2(\Omega))} \le C \tau^{1-\frac{1}{s}} \norm{\partial_t \u}_{L^s(I;L^2(\Omega))}.
	\]
	The spatial error is estimated by \cref{prop:estimate_Ritz_Stokes} resulting in
	\[
	\norm{\u - R_h^S (\u,p)}_{L^s(I;L^2(\Omega))} \le C h^2 \left(\norm{\nabla^2 \u}_{L^s(I;L^2(\Omega))} + \norm{\nabla p}_{L^s(I;L^2(\Omega))} \right).
	\]
Using maximal parabolic regularity and the convexity of $\Omega$, see \cref{remark:Omega_convex_1}, \cref{remark:s_to_infty} and \cref{cor:Omega_convex_2}, we obtain
\begin{equation}
\norm{\partial_t \u}_{L^s(I;L^2(\Omega))} +\norm{\nabla^2 \u}_{L^s(I;L^2(\Omega))} + \norm{\nabla p}_{L^s(I;L^2(\Omega))}
\le \frac{Cs^2}{s-1} \left(\norm{\f}_{L^s(I; L^2(\Omega))} + \norm{\u_0}_{\V^2}\right).
\end{equation}
For $s \ge 2$ we have $\frac{s^2}{s-1} \le 2s$. We choose $s = 2 \ln \frac{\finaltime}{{\tau}} \ge 2$ and get
\[
    \tau^{-\frac{1}{s}} =  \finaltime^{-\frac{1}{s}} \left(\frac{\finaltime}{\tau}\right)^{\frac{1}{s}} \le C(\finaltime) e^{\frac{1}{2}}. 
\]
Combining these terms we obtain the desired estimate.
\end{proof}
\begin{remark}
Under the additional assumption $\u_t$, $\Delta \u$, $\nabla p \in L^\infty(I,L^2(\Omega)^d)$, it is possible to remove one of the logarithmic terms in the result of \cref{theorem:error_estimate}.
\end{remark}

To compare our error estimate from \cref{theorem:error_estimate} with the results from the literature we first remark that our result especially holds for $w=0$, i.e., the dG($0$) discretization in time, which is known to be a variant of the implicit Euler scheme. In \cite{1986Heywood} the authors discuss the discretization of the transient Navier-Stokes equation by the implicit Euler scheme in time and finite elements in space. They prove an estimate of order ${\mathcal O}(\tau+h^2)$ (which corresponds to \cref{theorem:error_estimate} up to a logarithmic term) for the velocity error in the $L^\infty(I,L^2(\Omega)^d)$ norm, see  \cite[p. 765]{1986Heywood}. However, they require stronger regularity assumptions, in particular $\partial_t \f \in L^\infty(I;L^2(\Omega)^d)$. Note, that our setting and the setting from \cite{1986Heywood} are not fully comparable.

The authors of \cite{2010Chrysafinos} operate in a similar setting as here, discussing the discontinuous Galerkin method for the temporal discretization of the Stokes problem.
For the full discretization they derive the following estimate, see \cite[Theorem 4.9]{2010Chrysafinos} 
\begin{multline} \label{chap:IS:error_estimate_3}
    \norm{\u-\u_{{\tau}h}}_{L^{\infty}(I;L^2(\Omega))}
    \leq C \Big(  h\Big( \norm{\u}_{L^{2}(I;H^2(\Omega))} + h\norm{\u}_{L^{\infty}(I;H^2(\Omega))}\Big)
	+ {\tau} \Big(\norm{\u}_{H^{1}(I;H^1(\Omega))} + \norm{\u}_{W^{1,\infty}(I;L^2(\Omega))}\Big) \\
    \qquad +h\norm{\u_0}_{H^1(\Omega)} + \norm{\u}_{C(I;H^2(\Omega))}\min\Big(h^{3/2}/{\tau},\sqrt{h/{\tau}}\Big)h^{3/2} + h \norm{p}_{L^{2}(I;H^1(\Omega))}
\Big)
\end{multline}
with corresponding regularity assumptions on the solution $(\u,p)$. Comparing our result in \cref{theorem:error_estimate} with \cref{chap:IS:error_estimate_3}, we want to emphasize that we require much less regularity, provide a better convergence order with respect to $h$ and do not have any ``mixed terms" containing both $h$ and $\tau$ (apart from the logarithmic term). 

\section{Discrete regularity estimate for the pressure}\label{sec:pressure}
The  above results were so far solely focused on the velocity estimates. To provide estimates in tune with \cref{cor:Omega_convex_2} also in the discrete setting, we extend the results from \cref{sec:maximal_regularity} to the gradient of the pressure for certain finite element discretizations. 
These pressure estimates do not immediately lead  to best approximation results as in the velocity case above, but to our knowledge discrete regularity estimates for the pressure have not yet been reported.
In this section we assume the domain $\Omega$ to be convex and will use the estimate discussed in \cref{remark:delta_est}, i.e.,
\begin{equation}\label{eq:dmpr_infty_with_Laplace_h}
\max_{1\leq m \leq M} \norm{\partial_t \u_{{\tau}h}}_{L^\infty(I_m; L^2(\Omega))} + \norm{\Delta_h \u_{{\tau}h}}_{L^\infty(I;L^2(\Omega))}
+ \max_{1\leq m \leq M} \norm{{\tau}_m^{-1} [\u_{{\tau}h}]_{m-1}}_{L^2(\Omega)}
\leq C \ln \frac{\finaltime}{{\tau}} \norm{\mathbb{P}_h\f}_{L^\infty(I; L^2(\Omega))},
\end{equation}
in the setting of \cref{chap:IS:corollary:maximal_regularity_discrete}.

In the sequel we require the following proposition, which is given in \cite[Theorems 3.6, 4.1]{2013Guzman}.
\begin{proposition}\label{chap:IS:pressure_est}
    Let $\X_h\times M_h$ fulfill the inf-sup condition in \cref{chap02:eq:discrete_infsup} and the discrete pressure space fulfill the assumption $M_h\subset L^2_0(\Omega)\cap H^1(\Omega)$.
    Then, there holds
    \begin{equation}
	\sup_{\v_h \in \X_h, \v_h \neq \vec 0} \frac{(\nabla l_h,\v_h)}{\twonorm{\v_h}} \geq C \twonorm{\nabla l_h} \qquad \forall l_h \in M_h.
    \end{equation}
\end{proposition}
Note, that $\nabla l_h$ is well-defined for these finite element spaces. Finite element spaces that fulfill these assumptions are among others the space of Taylor-Hood finite elements or the space of the mini element.
The next theorem provides a discrete regularity estimate for the pressure.
\begin{theorem}\label{DMPR:Pressure}
   Let $1 \leq s \le  \infty$, $\Omega$ be convex, $\f \in L^s(I,L^2(\Omega)^d)$ and $\u_0 =0$. Let moreover the assumptions of \cref{chap:IS:pressure_est,remark:delta_est} be fulfilled. Let $(\u_{\tau h},p_{\tau h})\in X^w_{\tau}(\X_h\times M_h)$ be the solution to \cref{eq:spacetime_discretization}. Then there holds
    \begin{equation}
	\norm{\nabla p_{{\tau}h}}_{L^{s}(I;L^2(\Omega))}
	\leq C \ln \frac{\finaltime}{{\tau}} \norm{
	\f}_{L^{s}(I; L^2(\Omega))}.
    \end{equation}
\end{theorem}
\begin{proof}
    We first consider the case $s=\infty$.

    Let $\tilde t \in I_{\tilde m}$ for $1\leq \tilde m \leq M$
    and let $\theta(t)$ from the proof of \cref{chap:IS:theorem:L2Linfty_stability} be the regularized Dirac function supported in the interior of the time interval $I_{\tilde m}$ (cf. \cref{theta_def}) such that for $\tilde t \in I_{\tilde m}$ it holds by \cref{chap:IS:pressure_est} and integration by parts,
    \begin{align}	%
	\norm{\nabla p_{{\tau}h}(\tilde t)}_{L^2(\Omega)}
	\leq C \sup_{\v_h \in \X_h, \v_h \neq \vec 0} \frac{ (\nabla p_{{\tau}h}(\tilde t), \v_h)_{\Omega}}{\twonorm{\v_h}}
	&= C \sup_{\v_h \in \X_h, \v_h \neq \vec 0}  \frac{(\nabla p_{{\tau}h}, \theta\v_h)_{I\times\Omega}}{\twonorm{\v_h}}\\
	&= C \sup_{\v_h \in \X_h, \v_h \neq \vec 0}  \frac{(-p_{{\tau}h}, \theta \nabla \cdot \v_h)_{I\times\Omega}}{\twonorm{\v_h}}.\label{eq:p_max_reg_1}
    \end{align}
    Since $p_{{\tau}h}(\x)$ is in $X_{\tau}^w(L^2(\Omega))$, using the orthogonal projection $P_{\tau}$ from \eqref{eq:P_tau} %
    we have 
    \begin{equation}
	(p_{{\tau}h}, \theta \nabla \cdot \v_h)_{I\times\Omega} = (p_{{\tau}h}, P_{\tau}(\theta) \nabla \cdot \v_h)_{I\times\Omega}=(p_{{\tau}h}, \nabla \cdot (P_{\tau}(\theta)  \v_h))_{I\times\Omega}.
    \end{equation}
    Notice that since here $P_{\tau}(\theta(t))\in X_{\tau}^w(L^2(\Omega))$ and it being constant in space and $\v_h\in \X_h$ being constant in time, we have $P_{\tau}(\theta) \v_h \in X^w_{\tau}(\X_h)$.
    Thus, testing the weak formulation of the fully discrete Stokes problem in \cref{eq:spacetime_discretization} with $(P_\tau(\theta) \v_h, 0)$, we have    \begin{align}
	(p_{{\tau}h},\nabla\cdot (P_{\tau}(\theta) \v_h))_{I\times \Omega}
	&=\sum_{m=1}^M \langle \partial_t\u_{\tau h}, P_{\tau}(\theta) \v_h \rangle_{I_m\times \Omega} - (\Delta_h \u_{\tau h}, P_{\tau}(\theta) \v_h)_{I\times \Omega}\\
	&\qquad + \sum_{m=2}^M ([\u_{\tau h}]_{m-1},(P_{\tau}(\theta))_{m-1}^+ \v_h )_{\Omega} - (\f,P_{\tau}(\theta) \v_h)_{I\times \Omega},\label{chap:IS:eq:aux_1}
    \end{align}
    using that $\u_0 = \vec 0$ and the definition of $\Delta_h$ for the second term. 
    Using the H\"older inequality, we obtain the following estimate
    \begin{align}
	\abs{(p_{{\tau}h},\nabla\cdot P_{\tau}(\theta) \v_h 
	)_{I\times \Omega}}
	&\leq  
	\Bigg[  \max_{1\leq m \leq M} \norm{\partial_t \u_{{\tau}h}}_{L^{\infty}(I_m; L^2(\Omega))} 
	+ \norm{\Delta_h \u_{{\tau}h}}_{L^{\infty}(I;L^2(\Omega))} 
    +\norm{\f}_{L^{\infty}(I;L^2(\Omega))} \Bigg]\\
	&\times \norm{P_\tau(\theta) \v_h}_{L^{1}(I;L^2(\Omega))} +\max_{2\leq m \leq M}\norm{{\tau}_m^{-1} [\u_{{\tau}h}]_{m-1}}_{L^2(\Omega)}  \sum^M_{m=1} {\tau}_m \norm{(P_\tau(\theta))_{m-1}^+ \v_h}_{L^2(\Omega)} .\label{eq:p_max_reg_2}
    \end{align} 
    Applying the stability of $P_\tau$ \cref{eq:P_tau_stability} and the bound of $\theta(t)$ in the $L^1$ norm, we have
    \begin{equation}
	\norm{P_\tau(\theta) \v_h}_{L^{1}(I;L^2(\Omega))} \leq C \norm{\theta}_{L^1(I)}\twonorm{\v_h} \leq C\twonorm{\v_h}.\label{eq:p_max_reg_3}
    \end{equation}
    Since  $\theta$ is  supported in the interior of $I_{\tilde m}$ and $P_\tau(\theta)_m=0$ for all $m\neq \tilde m$, we obtain by
    $$
    |P_\tau(\theta)_{\tilde m}^+|\leq C\norm{\theta}_{L^{\infty}(I_{\tilde m})}\leq C \tau_{\tilde m}^{-1}
    $$ 
    (cf.\ \cite[(eq. A.2)]{1995Schatz}) and the second assumption on the time mesh that 
    \begin{equation}
	\tau_{\tilde m +1}\twonorm{P_\tau(\theta)_{\tilde m}^+ \v_h} \leq C \twonorm{\v_h}.\label{eq:p_max_reg_4}
    \end{equation}    
    By \cref{eq:dmpr_infty_with_Laplace_h} we obtain
    \begin{equation}
	\max_{1\leq m \leq M} \norm{\partial_t \u_{{\tau}h}}_{L^{\infty}(I_m; L^2(\Omega))} + \norm{\Delta_h \u_{{\tau}h}}_{L^{\infty}(I;L^2(\Omega))} +\max_{2\leq m \leq M}\norm{{\tau}_m^{-1} [\u_{{\tau}h}]_{m-1}}_{L^2(\Omega)}
	\leq C\ln \frac{\finaltime}{{\tau}}\norm{
	\f}_{L^{\infty}(I; L^2(\Omega))}.
    \end{equation}
    Collecting the estimates above, we established that for any $\tilde t\in I$
    \begin{equation}
	\norm{\nabla p_{{\tau}h}(\tilde t)}_{L^2(\Omega)}\le
	C \sup_{\v_h \in \X_h, \v_h \neq \vec 0}  \frac{(-p_{{\tau}h}, \theta \nabla \cdot \v_h)_{I\times\Omega}}{\twonorm{\v_h}}	\leq C\ln \frac{\finaltime}{{\tau}}\norm{
	\f}_{L^{\infty}(I; L^2(\Omega))}.\label{eq:L2for nabla p}
    \end{equation}

    Next we discuss the case $s=1$.
    Here, direct application of \cref{chap:IS:pressure_est} leads to a $\v_h$ that is dependent on time and thus cannot be separated from the time integral which leads to technical difficulties.
    Thus, we will pursue a similar approach as above.
    We can expand the norm as follows
    \begin{align}
	\norm{\nabla p_{\tau h}}_{L^1(I;L^2(\Omega))} 
	= \sum_{m=1}^M \int_{I_m}\norm{\nabla p_{\tau h}}_{L^2(\Omega)}dt
	\leq \sum_{m=1}^M \tau_m \norm{\nabla p_{\tau h}}_{L^{\infty}(I_m;L^2(\Omega))}.
    \end{align}
    Similarly to the case $s=\infty$, using regularized Dirac functions $\theta^m(t)$ from the proof of \cref{chap:IS:theorem:L2Linfty_stability} for $\tilde t_m \in I_m$, we have
    \begin{align}
	\sum_{m=1}^M \tau_m \norm{\nabla p_{\tau h}(\tilde t_m)}_{L^2(\Omega)}
	\leq C \sum_{m=1}^M \tau_m \sup_{\v_h^m \in \X_h, \v_h^m \neq \vec 0}\frac{ (\nabla p_{\tau h}(\tilde t_m), \v_h^m)_{\Omega}}{\twonorm{\v_h^m}}
	= C \sum_{m=1}^M \tau_m \sup_{\v_h^m \in \X_h, \v_h^m \neq \vec 0}\frac{(\nabla p_{\tau h}, \theta^m\v_h^m)_{I_m \times \Omega}}{\twonorm{\v_h^m}}, \label{eq:aux_est_1}
    \end{align}
    where in the last step we used that $\theta^m$ is supported in $I_m$.
    Since  $\v_h^m$ and $\tau_m$ are constants on each $I_m$,  we can pull the supremum out of the sum, to obtain
    \begin{align}
	\sum_{m=1}^M \tau_m \norm{\nabla p_{\tau h}(\tilde t_m)}_{L^2(\Omega)}
	&\le C \sum_{m=1}^M  \sup_{\v_h^m \in \X_h, \v_h^m \neq \vec 0}\left(\nabla p_{\tau h}, \frac{\theta^m\tau_m\v_h^m}{\twonorm{\v_h^m}}\right)_{I_m \times \Omega}\\
	&= C \sum_{m=1}^M  \sup_{\v_h^m \in \X_h, \v_h^m \neq \vec 0}\left(\nabla p_{\tau h}, \frac{P_\tau(\theta^m)\tau_m\v_h^m}{\twonorm{\v_h^m}}\right)_{I_m \times \Omega}\\
	&= C \sup_{\substack{\v_h^1\in \X_h,\cdots,\v_h^M \in \X_h\\ \v_h^1\neq\vec 0,\cdots,\v_h^M \neq \vec 0}} \left(\nabla p_{\tau h}, \sum_{m=1}^M \frac{P_\tau(\theta^m)\tau_m\v_h^m}{\twonorm{\v_h^m}}\right)_{I \times \Omega}\\
	&= C \sup_{\substack{\v_h^1\in \X_h,\cdots,\v_h^M \in \X_h \\ \v_h^1\neq\vec 0,\cdots,\v_h^M \neq \vec 0}}
	(\nabla p_{\tau h},\tilde \v_{\tau h})_{I \times \Omega},
    \end{align}
    where we defined $\tilde \v_{\tau h}\in  X^w_{\tau}(\X_h)$ by
    $$
    \tilde \v_{\tau h} :=\sum_{m=1}^M \frac{P_\tau(\theta^m)\tau_m\v_h^m}{\twonorm{\v_h^m}}.
    $$
    Using the weak formulation in \cref{eq:spacetime_discretization} and the H\"older estimate as before, we see
    \begin{align}
	\abs{(\nabla p_{\tau h},\tilde \v_{\tau h})_{I \times \Omega}}
	 &\leq C 
	 \Bigg[  \sum_{m=1}^M \norm{\partial_t \u_{{\tau}h}}_{L^1(I_m; L^2(\Omega))} + \norm{\Delta_h \u_{{\tau}h}}_{L^{1}(I;L^2(\Omega))}
	 + \sum_{m=2}^M \tau_m\norm{{\tau}_m^{-1} [\u_{{\tau}h}]_{m-1}}_{L^2(\Omega)} \\
	 &\qquad +  \norm{\f}_{L^1(I;L^2(\Omega))} \Bigg]
	 \times \Bigg[  \max_{1\leq m \leq M}  \norm{\tilde \v_{{\tau}h,m-1}^+}_{L^2(\Omega)}  + \norm{\tilde \v_{{\tau}h}}_{L^{\infty}(I;L^2(\Omega))}\Bigg].
    \end{align}
    By \cref{remark:delta_est} and \cref{chap:IS:corollary:maximal_regularity_discrete} for $s=1$ we have, similar to \cref{eq:dmpr_infty_with_Laplace_h}
    \begin{equation}
	\sum_{m=1}^M \norm{\partial_t \u_{{\tau}h}}_{L^1(I_m; L^2(\Omega))} + \norm{\Delta_h \u_{{\tau}h}}_{L^{1}(I;L^2(\Omega))}
	+ \sum_{m=2}^M \tau_m\norm{{\tau}_m^{-1} [\u_{{\tau}h}]_{m-1}}_{L^2(\Omega)} 
	\le  C\ln \frac{\finaltime}{{\tau}}\norm{\f}_{L^1(I;L^2(\Omega))}.
    \end{equation}
    Using the stability of $P_\tau$ in $L^\infty(I)$ \cref{eq:P_tau_stability}, we obtain
    \begin{align}
	\norm{\tilde \v_{{\tau}h}}_{L^{\infty}(I;L^2(\Omega))}
	= \max_{1\leq m \leq M}\norm{\tilde \v_{{\tau}h,m}}_{L^{\infty}(I_m;L^2(\Omega))}
	&= \max_{1\leq m \leq M}
	\frac{\|P_\tau(\theta^m)\tau_m\v_h^m\|_{L^{\infty}(I_m;L^2(\Omega))}
	}{\twonorm{\v_h^m}}\\
	&\le \max_{1\leq m \leq M} \frac{C \norm{\theta^m}_{L^{\infty}(I_m)}\tau_m\twonorm{\v_h^m}}{\twonorm{\v_h^m}}
	\leq \max_{1\leq m \leq M} C,
    \end{align}
    where we used that $\|\theta^m\|_{L^\infty(I_m)}\le C{\tau_m}^{-1}$. Similarly, $\max_{1\leq m \leq M}  \norm{\tilde \v_{{\tau}h,m-1}^+}_{L^2(\Omega)}\le C$.
    Combining the steps above, we arrive at the following estimate:
    \begin{equation}
	\norm{\nabla p_{\tau h}}_{L^1(I;L^2(\Omega))}
	\leq \sum_{m=1}^M \tau_m \norm{\nabla p_{\tau h}}_{L^{\infty}(I_m;L^2(\Omega))}\leq C \ln \frac{\finaltime}{{\tau}}\norm{\f}_{L^1(I;L^2(\Omega))}.
    \end{equation}
    This shows the estimate in $L^1(I;L^2(\Omega))$. By interpolation we obtain the result for $1\leq s\leq \infty$.
\end{proof}

\begin{corollary}
    Let $p_{{\tau}h}$ be the pressure solution to \cref{eq:spacetime_discretization} with $\f=\vec 0$. Let moreover the assumptions of \cref{chap:IS:pressure_est,remark:delta_est} be fulfilled. Then, there holds for $m= 1,2, \dots, M$
    \begin{equation}
	\norm{\nabla p_{{\tau}h}}_{L^{\infty}(I_m;L^2(\Omega))} \leq \frac{C}{t_m} \norm{
	\u_0}_{L^2(\Omega)}.
    \end{equation}
\end{corollary}
\begin{proof}
    The result follows by the same arguments which we used to show the theorem above. 
    A notable difference is that we only consider $I_m$ here, not the whole domain and use \cref{chap:IS:theorem:maximal_regularity_smoothing_discrete} instead of \cref{chap:IS:corollary:maximal_regularity_discrete}. 
\end{proof}

\section*{Funding}
Deutsche Forschungsgemeinschaft (German Research Foundation) - Project number 188264188/GRK1754 to N.B.;
NSF grant DMS-1913133 to D.L.

\bibliographystyle{IMANUM-BIB}
\bibliography{references}

\clearpage

\appendix

\end{document}